\documentclass[11pt]{amsart}
\usepackage{latexsym, amssymb, url, mathrsfs, amsfonts, MnSymbol}
\usepackage[pdftex,lmargin=1.45in, rmargin=1.45in,tmargin=1.25in,bmargin=1.25in, marginpar=3.25cm]{geometry}

\usepackage{color}

\renewcommand{\MR}[1]{ }

\theoremstyle{plain}

\newtheorem{innercustomthm}{Theorem}
\newenvironment{customthm}[1]
  {\renewcommand\theinnercustomthm{#1}\innercustomthm}
  {\endinnercustomthm}

\newtheorem{thm}{Theorem}
\newtheorem{conj}[thm]{Conjecture}
\numberwithin{thm}{subsection}

\newtheorem{coro}[thm]{Corollary}
\newtheorem{lem}[thm]{Lemma}
\newtheorem{prop}[thm]{Proposition}

\theoremstyle{definition}

\newtheorem{defi}[thm]{Definition}

\theoremstyle{remark}
\newtheorem{remark}[thm]{Remark}
\newtheorem{rmk}[thm]{Remark}

\title[]{Differential operators mod $p$:\\
analytic continuation and consequences}
\author[Eischen]{E. E. Eischen$^\dag$}
\thanks{$\dag$Partially supported by NSF Grants DMS-1559609 and DMS-1751281.}
\author[Flander]{M. Flander$^\ddag$}
\thanks{$\ddag$Supported by an Australian Postgraduate Award.}
\author[Ghitza]{A. Ghitza}
\author[Mantovan]{E. Mantovan}
\author[McAndrew]{A. McAndrew$^*$}
\thanks{$*$Partially supported by an Australian Postgraduate Award and the Albert Shimmins Writing Up Award.}

\address{E. E. Eischen\\
Department of Mathematics\\
University of Oregon\\
Fenton Hall\\
Eugene, OR 97403-1222\\
USA}
\email{eeischen@uoregon.edu}

\address{M. Flander\\
PlanGrid\\
2111 Mission St.\ Suite 400\\
San Francisco, CA 94110\\
USA
}
\email{max.flander@gmail.com}

\address{A. Ghitza\\
School of Mathematics and Statistics\\
University of Melbourne\\
Parkville VIC 3010\\
 Australia\\
}
\email{aghitza@alum.mit.edu}

\address{E. Mantovan\\
Department of Mathematics\\
Caltech\\
1200 E California Blvd\\
Pasadena, CA 91125\\
USA
}
\email{mantovan@caltech.edu}

\address{A. McAndrew\\
Department of Mathematics and Statistics\\
Boston University\\
111 Cummington Mall\\
 Boston MA 02215\\
USA
}
\email{angusmca@bu.edu}

\date{\today}

\renewcommand{\geq}{\geqslant}
\renewcommand{\leq}{\leqslant}

\newcommand{\Aut}{\operatorname{Aut}}

\newcommand{\RR}{\mathbb{R}}

\newcommand{\Fbar}{\overline{F}}
\newcommand{\FF}{\mathbb{F}}
\newcommand{\fp}{\mathbb{F}_p}
\newcommand{\fpbar}{\overline{\FF}_p}
\newcommand{\Frob}{\operatorname{Frob}}
\newcommand{\Gal}{\operatorname{Gal}}
\newcommand{\gm}{\mathbb{G}_m}
\newcommand{\GSp}{\operatorname{GSp}}

\newcommand{\longto}{\longrightarrow}

\newcommand{\oh}{\mathscr{O}}

\newcommand{\Tr}{\operatorname{Tr}}

\newcommand{\xub}{X^\bullet}
\newcommand{\xlb}{X_\bullet}
\newcommand{\dub}{\Delta^\bullet}
\newcommand{\dlb}{\Delta_\bullet}

\newcommand{\CA}{\mathcal{A}}

\newcommand{\Auniv}{\CA}

\newcommand{\Sh}{\mathcal{X}}

\newcommand{\Shimuravariety}{\Sh}

\newcommand{\Shp}{X}

\newcommand{\shord}{{\mathcal S}}

\newcommand{\Shpord}{S}

\newcommand{\shpord}{S}

\newcommand{\hdrp}{{H^1_{\mathrm dR}(\Auniv/\Shp)}}

\newcommand{\witt}{W(\overline{\bF}_p)}
\newcommand{\Witt}{W(\overline{\mathbb F}_p)}
\newcommand{\gf}{\Gal\left(\Fbar/F\right)}
\newcommand{\gq}{\Gal\left(\Qbar/\QQ\right)}

\newcommand{\WittZp}{\mathbb{W}}

\newcommand{\ee}{\mathbf{e}}

\newcommand{\Trace}{\mathrm{Trace}}

\newcommand{\ord}{\mathrm{ord}}

\newcommand{\CO}{\mathcal{O}}

\newcommand{\Image}{\mathrm{Image}}

\newcommand{\KS}{\mathrm{KS}}
\newcommand{\ks}{\mathrm{ks}}
\newcommand{\id}{\mathrm{id}}
\newcommand{\cmfield}{F}

\newcommand{\CE}{\mathcal{E}}

\newcommand{\TT}{\mathcal{T}}

\newcommand{\Lie}{\mathrm{Lie}}

\newcommand{\Spec}{\mathrm{Spec}}

\newcommand{\Hom}{\mathrm{Hom}}

\newcommand{\diag}{\mathrm{diag}}

\newcommand{\Z}{{\mathbb Z}}

\newcommand{\cF}{\mathcal{F}}

\newcommand{\cG}{\mathcal{G}}

\newcommand{\C}{{\mathbb C}}

\newcommand{\T}{\mathcal{T}}
\newcommand{\G}{\mathcal{G}}

\newcommand{\p}{\mathfrak{p}}

\newcommand{\hasse}{{E}}

\newcommand{\bF}{{\mathbb F}}
\newcommand{\cO}{{\mathcal O}}

\newcommand{\curlyU}{\mathscr{U}}

\newcommand{\absfrob}{\mathrm{F}}

\newcommand{\relfrob}{\mathrm{Fr}}

\newcommand{\hdrord}{{H^1_{\mathrm dR}(\Auniv/\shord)}}

\newcommand{\hdrordp}{{H^1_{\mathrm dR}(\Auniv/\shpord)}}

\newcommand{\rpi}{{R^1\pi_*\cO}}

\newcommand{\ci}{C^{\infty}}

\newcommand{\IQ}{\QQ}
\newcommand{\Qbar}{\overline{\QQ}}

\newcommand{\GL}{\mathrm{GL}}

\newcommand{\adeles}{\mathbb{A}}

\newcommand{\ZZ}{\mathbb{Z}}
\newcommand{\isomto}{\overset{\sim}{\rightarrow}}
\newcommand{\CC}{\mathbb{C}}

\newcommand{ \Diff}{{\mathcal D}}

\newcommand{\pr}{{\pi}}

\newcommand{\PR}{{\Pi}}

\newcommand{\yo}{y}

\newcommand{\adj}[1]{\,^{\star}\!\!#1}

\newcommand{\QQ}{\mathbb{Q}}

\newcommand{\varla}{d}

\newcommand{\hdrAY}{H^1_{\mathrm dR}(\mathscr{A}/Y)}

\renewcommand{\L}[1]{\,^{L}\!#1}

\newcommand{\zqv}{\ZZ\left[q_v^{\pm 1/2}\right]}

\renewcommand{\AA}{\mathbb{A}}

\begin{document}

\bibliographystyle{amsalpha}

\begin{abstract}
This paper concerns certain $\mod p$ differential operators that act on automorphic forms over Shimura varieties of type A or C.  We show that, over the ordinary locus, these operators agree with the $\mod p$ reduction of the $p$-adic theta operators previously studied by some of the authors.  In the characteristic $0$, $p$-adic case, there is an obstruction that makes it impossible to extend the theta operators to the whole Shimura variety.  On the other hand, our $\mod p$ operators extend (``analytically continue'', in the language of de Shalit and Goren) to the whole Shimura variety.  As a consequence, motivated by their use by Edixhoven and Jochnowitz in the case of modular forms for proving the weight part of Serre's conjecture, we discuss some effects of these operators on Galois representations.

Our focus and techniques differ from those in the literature.
Our intrinsic, coordinate-free approach removes difficulties that arise from working with $q$-expansions and works in settings where earlier techniques, which rely on explicit calculations, are not applicable.  In contrast with previous constructions and analytic continuation results,
these techniques work for any totally real base field, any weight, and all signatures and ranks of groups at once, recovering prior results on analytic continuation as  special cases.

\end{abstract}

\maketitle
\vspace{-0.25in}

\setcounter{tocdepth}{1}

\section{Introduction}\label{intro-section}

Fix an odd prime number $p$.  This paper concerns certain mod $p$ differential operators that act on automorphic forms over Shimura varieties of types A and C.  Key motivation comes from the setting of modular forms and associated Galois representations.  Even though some of the ingredients present in the case of modular forms have been established in broad generality, extending them to higher rank setting involves various challenges.
\subsection{Context from the special case of $\GL_2$}\label{GL2-intro}
In 1977, Katz constructed a mod $p$ differential operator $\theta$ that acts on mod $p$ modular forms~\cite{Katz-theta}.  Given a mod $p$ modular form $f$, the operator $\theta$ acts on the $q$-expansion $f(q)= \sum_{n\geq 0} a_n q^n$ by
\begin{align}\label{qformula}
\theta f(q) = \sum_{n\geq 0} na_n q^n.
\end{align}
Mod $p$ modular forms of weight $k$ and level $N$, with $p\ndivides N$, arise as global sections of a certain line bundle $\omega^{\otimes k}$ over a curve $\mathcal{M}_N$ parametrizing elliptic curves with level $N$ structure.  That is, modular forms of weight $k$ and level $N$ are elements of $H^0\left(\mathcal{M}_N, \omega^{\otimes k}\right)$.  The operator $\theta$ is, \emph{a priori}, constructed as a map $H^0\left(\mathcal{M}^\ord_N, \omega^{\otimes k}\right)\rightarrow H^0\left(\mathcal{M}^\ord_N, \omega^{\otimes k+2}\right)$, i.e.\ only over the ordinary locus $\mathcal{M}^\ord_N$ of $\mathcal{M}_N$.  Given any modular form $f\in H^0\left(\mathcal{M}_N, \omega^{\otimes k}\right)$, however, Katz explains how to extend (``analytically continue'', in the terminology of~\cite{DSG2}) the section $\theta\left(f|_{\mathcal{M}^{\ord}_N}\right)$ to a classical modular form of weight $k+p+1$, i.e.\ an element of $H^0\left(\mathcal{M}_N, \omega^{\otimes k+p+1}\right)$ whose $q$-expansions are the same as those of $\theta\left(f|_{\mathcal{M}^{\ord}_N}\right)$.

 Associated to a cuspidal eigenform $f$ with Fourier expansion $f(q) = \sum_{n\geq 1}a_n q^n$, normalized so that $a_1 = 1$, there is a continuous, semi-simple representation
\begin{align}\label{char0repn}
  \rho_f\colon \gq\rightarrow \GL_2\left(\Qbar_p\right),
\end{align}
unramified at all primes $\ell$ coprime to $p$ and to the level of $f$, and satisfying
\begin{align*}
\Trace\left(\rho_f\left(\Frob_\ell\right)\right) = a_\ell
\quad\text{and}\quad 
\det\left(\rho_f\left(\Frob_\ell\right)\right) = \ell^{k-1}\psi(\ell)
\end{align*}
at such primes, where $k$ is the weight of $f$ and $\psi$ is the nebentypus of $f$~\cite{eichler, shimura1, shimura2, igusa, deligne1, deligne2, deligne-serre}. 
 Equivalently, using the description of Hecke operators on $q$-expansions, this statement can be reformulated in terms of Hecke eigenvalues in place of Fourier coefficients.
Serre's conjectures (formulated in~\cite{serreconj1, serreconj2} and proved in generality in~\cite{KW1, KW2, KW3} and in certain cases in~\cite{dieulefait1, dieulefait2}) associate to a representation
\begin{align*}
\rho\colon \gq\rightarrow \GL_2\left(\FF\right),
\end{align*}
with $\FF$ a finite field of characteristic $p$, a cuspidal eigenform $f$ such that $\overline{\rho}_f\cong \rho$, where $\overline{\rho}_f$ denotes the reduction  of $\rho_f$ mod $p$.
 Serre's conjectures also specify the weight and level of $f$.

The effect of iterations of $\theta$ (``theta cycles'') plays a key role in the proof of the weight part of Serre's conjecture~\cite{Edixhoven, jochnowitz}.
Via the action of Hecke operators $T_\ell$ on $q$-expansions, it is simple to observe that
\begin{align}\label{commrel-equ}
T_\ell \theta = \ell \theta T_\ell.
\end{align}
As a consequence, the effect of $\theta$ on the mod $p$ Galois representation attached to a mod $p$ eigenform $f$ is
\begin{align}\label{twist-equ}
\rho_{\theta f}  = \chi\otimes\rho_f,
\end{align}
where $\chi\colon \gq\rightarrow \FF_p^\times$ is the mod $p$ cyclotomic character.

Continuing with the case of $\GL_2$, Andreatta and Goren studied $\mod p$ theta operators in the case of Hilbert modular forms over a totally real field $L$. They also obtained several results on theta cycles in~\cite[Section 16]{AndreattaGoren}, building on earlier work of Jochnowitz~\cite{jochnowitz}, at least assuming $p$ remains inert in $L$ and the weight is parallel (although they also obtain preliminary results in the non-parallel case).

\subsection{The case of higher rank (especially unitary and symplectic) groups}\label{intro-newresults}
Given recent developments for automorphic forms on higher rank groups and given the significance of mod $p$ theta operators in the case of modular forms, it is natural to try to extend results on mod $p$ theta operators to the setting of higher rank groups.  In particular, one desires:
\begin{enumerate}
\item{A mod $p$ theta operator that is defined over the whole Shimura variety (and, ideally, the relationship with the mod $p$ reduction of the $p$-adic operators constructed over the ordinary locus, e.g.\ in~\cite{EDiffOps, EFMV, EiMa}).}
\item{Knowledge of the amount by which this mod $p$ theta operator increases the weight of an automorphic form.}
\item{Connections with Galois representations.}
\end{enumerate}

Since various ingredients exist in broad generality for higher dimensional groups, it is tempting to think one can handle all settings at once.  In fact, the arguments in the literature are delicate and have involved imposing restrictions specific to the cases at hand. 
To give the reader an idea of the state of the art prior to this paper, we briefly summarize the most recent developments (primarily from the past year):
\begin{itemize}
\item{{\bf Siegel modular forms (symplectic groups):} Yamauchi recently produced results for $\mod p$ Siegel modular forms, in the case of $\GSp_4$ over $\IQ$~\cite{Yama}.  His clever computations rely on the numerics of the case at hand, including conditions on the weights, and do not naturally generalize.}
\item{{\bf Unitary groups:} Via techniques relying on Fourier--Jacobi expansions, de Shalit and Goren achieved analytic continuation when the base real field is $\IQ$ (so the CM field associated to the unitary group is imaginary quadratic) and the weight is scalar in~\cite{DSG, DSG2}.  It is not immediately clear how to use this approach when the degree of the field extension is greater or the weight is non-scalar.  Furthermore, under these conditions, they achieved new, interesting results on theta cycles, which require that $p$ be inert in the imaginary quadratic field and the signature of the unitary group to be $(n, 1)$, and the numerics there are particular to these conditions being met.}
\end{itemize}

Our focus and techniques differ from those recently introduced in the literature, and the principal achievements of the present paper include removal of some of these conditions for construction and analytic continuation of theta operators.  Our intrinsic, coordinate-free approach avoids a need for $q$-expansions and works in settings where their techniques, which include explicit calculations, are not applicable.
In particular, for analytic continuation, our techniques immediately:
\begin{itemize}
\item{Work for any totally real base field $E$ (while their approaches require $E=\IQ$).}
\item{Allow all vector and scalar weights (in contrast to the restrictions above).}
\item{Handle all signatures and ranks of groups at once (whereas Yamauchi employs $\GSp_4$ and de Shalit--Goren's approach divides signatures into two cases).} 
 \end{itemize}
 When the ordinary locus is nonempty and their constraints are also met, our operators agree with theirs.  As explained in~\cite{EiMa2}, the present approach can also be extended (with additional work) to the case where the ordinary locus is empty.
Now that we have defined this operator in much more generality, we show that results of Katz from the case of modular forms (which can be viewed as preliminary results toward theta cycles) extend to Hermitian (i.e.\ unitary of signature $(n,n)$) and Hilbert--Siegel modular forms when the weight is parallel scalar but the base field can be of arbitrary degree.  We also achieve other nearly immediate results for Galois representations and set the stage for further developments in~\cite{EiMa2}.

Our proofs are entirely intrinsic to the geometry of the underlying Shimura varieties.   This allows us to overcome what might appear---from the setup for modular forms in Section~\ref{GL2-intro}---to be an obstruction to working with more general groups: We do not necessarily have $q$-expansions in our setting, thus depriving us of analogues of Equation~\eqref{qformula}.  By exploiting geometry, however, we obtain clean, intrinsic descriptions without reference to $q$-expansions (and without reference to Serre--Tate or Fourier--Jacobi expansions, which might at first appear to be reasonable substitutes but turn out not to be similarly suitable for studying the action of the Hecke algebras).

Our first result below concerns operators $\Diff^\lambda$, generalizing $\theta$, that have been previously constructed over the ordinary locus of unitary and symplectic Shimura varieties (e.g.\ in~\cite{EDiffOps, EFMV}, when the ordinary locus is nonempty, building on ideas introduced by Katz and Harris in~\cite{kaCM, hasv, ha86}). More precisely, in~\cite{EDiffOps, EFMV}, the operators $\Diff^\lambda$ are defined on $p$-adic forms over the ordinary locus. Here, we are concerned with their reduction modulo $p$, which provides an analogous weight-raising operator for $p$ sufficiently large (as explained in Remark~\ref{p>n} of the present paper).

Note that, in contrast to the conditions forced by the approaches of prior work in the literature, there are no restrictions on the degree of the base field and no restrictions on the weight of the automorphic forms here. As a consequence of Theorem~\ref{ana_thm} of the present paper, we obtain the following result.

\begin{customthm}{A}[Analytic continuation of mod $p$ differential operators]\label{continuationB}
In the symplectic and unitary cases, the differential operators $\Diff^\lambda$, defined \emph{a priori} only over the ordinary locus, can be analytically continued to the whole mod $p$ Shimura variety to give a differential operator $\Theta^\lambda$ on mod $p$ automorphic forms.
These operators specialize in the case of $\GSp_2= \GL_2$ to Katz's operator $\theta$ described in Equation~\eqref{qformula}.  
\end{customthm}

Similarly to Katz's case of raising the weight by $p+1$, our operators also raise the weight as specified in Theorem~\ref{charpcor}.  With these operators, we can reformulate Equations~\eqref{commrel-equ} and~\eqref{twist-equ} to obtain Theorem~\ref{charpcor}.  (Below, $\hat{G}$ denotes the dual group of $G$, $\hat{\nu}$ is the cocharacter dual to the similitude character $\nu$ of $G$, and  $\underline{k}$ is the parallel weight with entries all equal to $k\in\ZZ$.)

\begin{customthm}{B}[Action of differential operators on mod $p$ Galois representations]\label{charpcor}
  Let $G$ be a symplectic or unitary group over $\QQ$, split over a number field $F$.
Let $f$ be a mod $p$ Hecke eigenform on $G$ of weight $\kappa$, and let $\lambda$ be a weight that is admissible (as in Definition~\ref{admissible-defi}).  Assume $\Theta^\lambda(f)
$ is nonzero.  
Then $\Theta^\lambda (f)$ is a Hecke eigenform of weight $\kappa+\lambda+\underline{(p-1)|\lambda/2|}$. 

Furthermore, for  $\rho\colon\gf\to\hat{G}\left(\fpbar\right)$ a continuous representation, the Frobenius eigenvalues of $\rho$ agree with the Satake parameters of $f$ (as defined in Conjecture~\ref{conj:galois}) if and only if the Frobenius eigenvalues of
$(\hat{\nu}^{|\lambda|/2}\circ \chi)\otimes \rho$
agree with the Satake parameters of $\Theta^\lambda(f)$.   In particular, if $\rho$ is modular of weight $\kappa$, then $(\hat{\nu}^{|\lambda|/2}\circ \chi)\otimes \rho$ is modular of weight $\kappa+\lambda+\underline{(p-1)|\lambda/2|}$.
\end{customthm}

\begin{rmk}
  The reader who prefers to work with representations of $\gq$ instead of $\gf$ can do so by replacing the split version of the Satake isomorphism described in Section~\ref{AG-bkgd} (where we follow~\cite{Gross-satake}) with the non-split version from~\cite[Section 7]{treumannvenkatesh}.
  The Galois representations are then of the form $\gq\to\L{G}\left(\fpbar\right)$, taking values in the Langlands dual group of $G$.
\end{rmk}

\begin{rmk}
  In the classical $\GL_2$ case described in Section~\ref{GL2-intro}, passing from the commutation relation between differential and Hecke operators (Equation~\eqref{commrel-equ}) to the twisting effect on Galois representations (Equation~\eqref{twist-equ}) is immediate, given the well-known expression of the characteristic polynomial of Frobenius as an explicit quadratic polynomial.
  The same approach can be used in the $\GSp_4$ case~\cite{Yama}, where the characteristic polynomial of Frobenius can be written down as an explicit quartic polynomial.
  Beyond these two low rank cases, the explicit approach is impractical given that the polynomial for $\GSp_{2g}$ has degree $2^g$.
  We therefore adopt a more conceptual view and use the Satake isomorphism to establish the effect of differential operators on Galois representations in Theorem~\ref{charpcor}.
\end{rmk}

\begin{rmk}\label{katzfiltratiion-rmk}
  Katz discusses the \emph{exact filtration} of a $\bmod p$ modular form $f$~\cite[Section I]{Katz-theta}: the smallest $k$ such that $f$ is not divisible by the Hasse invariant, or equivalently, the smallest $k$ such that there is no modular form of weight $k'<k$  whose $q$-expansion at some cusp agrees with the $q$-expansion of $f$.  The notion of exact filtration is related to questions about the kernel of the differential operator (times the Hasse invariant) and to the proof of the weight part of Serre's conjecture, as described in, e.g.~\cite[Section 3]{Edixhoven}.  As an indication of potential future applications of Theorem~\ref{charpcor}, Section~\ref{final-effects-section} contains some remarks on the effect of the differential operator on the exact filtration, in the special case of Siegel or Hermitian forms of parallel weight.
\end{rmk}

\begin{rmk}
As noted above, when working with unitary groups, we assume the ordinary locus is nonempty.  Fully extending our results to the case where the ordinary locus is empty, the \emph{$\mu$-ordinary} setting, requires substantial additional work and is the subject of~\cite{EiMa2}.  The $\mu$-ordinary case is also addressed, via different techniques, when the base field is $\IQ$ (so the field associated to the unitary group is quadratic imaginary) and the weight is scalar in~\cite{DSG, DSG2}.
\end{rmk}

\begin{rmk}
The differential operator $\theta$ on classical modular forms can be constructed via several other methods, some of which have been generalized to certain groups of higher rank.  The interested reader can consult the survey in~\cite{ghitza_rims}.
\end{rmk}

\subsection{Obstructions to analytic continuation in characteristic $0$}\label{char0-section}

Since the representation~\eqref{char0repn} lives in characteristic $0$, it is natural to ask about the possibility of lifting Theorems~\ref{continuationB} and~\ref{charpcor} from the mod $p$ setting to the $p$-adic setting.

As an intermediate step, one might try to lift to characteristic $p^m$ for $m>1$ an integer. 
Chen and Kiming have accomplished this for mod $p^m$ modular forms $f$, and they have proved that if $f$ has weight filtration $k$, then $\theta f$ has weight filtration $k+2+2p^{m-1}(p-1)$~\cite[Theorem 1]{chen-kiming}.  This shows that, for modular forms over $\ZZ_p$, there is no way to analytically continue the operator $\theta$ from the ordinary locus to the entire Shimura variety over $\ZZ_p$.  Indeed, if it were possible to analytically continue $\theta$, then the weight filtration of $\theta f$ mod $p^m$ would be bounded above by the weight of the characteristic $0$ form $\theta f$; but as Chen and Kiming have shown, the weight filtration of $\theta f$ mod $p^m$ is unbounded as $m$ goes to infinity.

Observing this obstruction to analytically continuing the entire form in characteristic $0$, one might try just to analytically continue the \emph{ordinary} part of $\Diff_{\kappa}^\lambda f$.  More precisely, let $\ee$ denote Hida's ordinary projector (defined in~\cite[Sections 6--8]{H02}), which acts on the space of $p$-adic automorphic forms.  
A $p$-adic automorphic form is \emph{ordinary} if it is in the image of $\ee$.  Every ordinary $p$-adic automorphic form of sufficiently regular
 algebraic weight is, in fact, an ordinary classical automorphic form defined over the whole Shimura variety $\Shimuravariety$, by~\cite[Theorems 6.8(4) and 7.1(4)]{H02}.  So even though $\Diff_{\kappa}^\lambda f$ cannot be extended to all of $\Shimuravariety$, we at least have that $\ee\Diff_{\kappa}^\lambda f$ extends to all of $\Shimuravariety$.

As we shall see in Corollary~\ref{ed0}, though, the proof of Theorem~\ref{charpcor} (an analysis of the interaction between Hecke operators and differential operators) also shows that $\ee\Diff_{\kappa}^\lambda = 0$.  So the ordinary part of $\Diff_{\kappa}^\lambda f$ is $0$.  In the case of modular forms in characteristic $0$, Coleman, Gouv\^ea, and Jochnowitz also showed in~\cite[Corollary 10]{CGJ} that $\theta$ destroys overconvergence, due to a relationship with the weight $2$ Eisenstein series $E_2$ (a relationship that actually enables the mod $p^m$ liftings in~\cite{chen-kiming}).

\begin{rmk}
Readers who are primarily familiar with the characteristic $0$, $p$-adic setting should be careful not to equate the objectives of this paper with recent work on overconvergence, which concerns extending modular forms past the ordinary locus (but not to the whole moduli space) as a step toward extending certain constructions of $p$-adic $L$-functions.  As the above shows, in characteristic 0, there is no hope of extending our operator to the whole Shimura variety.  Furthermore, recent work on overconvergence and theta operators has primarily focused on $\GL_2$, whereas we are interested in higher rank groups. \end{rmk}

\subsection{Organization}
The meat of the paper is in Sections~\ref{analytic_sec} and~\ref{Hecke_sec}.  First, though, Section~\ref{prelim_sec} introduces necessary conventions and background (i.e.\ previously established results and machinery on which our present paper relies), divided into two parts.    The  first focuses on Hecke algebras and Galois representations. The second provides  background on Shimura varieties and automorphic forms (including Hasse invariants, a key ingredient in the analytic continuation of differential operators).

Section~\ref{analytic_sec} establishes the analytic continuation of the modulo $p$ reduction of $p$-adic differential operators.  Those operators were previously constructed over the ordinary locus in~\cite{EDiffOps, EFMV}.  The key work in this section lies in producing another construction of differential operators in characteristic $p$ that is defined over the whole Shimura variety and then showing that it agrees with the mod $p$ reduction of the $p$-adic differential operators we originally constructed.  Theorem~\ref{continuationB} is a consequence of the main result of this section, Theorem~\ref{ana_thm}, which demonstrates the role of Hasse invariants in analytically continuing differential operators, in analogy with Katz's approach in~\cite{Katz-theta}.  

Section~\ref{Hecke_sec} describes the interaction between differential operators and Hecke operators.  Our proofs rely entirely on the geometry of Shimura varieties and do not require $q$-expansions.
Theorem~\ref{TandTheta} shows that the differential operators commute with prime-to-$p$ Hecke operators up to scalar multiples, in analogy with Equation~\eqref{commrel-equ}, a key step toward describing the effect on Galois representations.  
  
Section~\ref{galois-section} establishes consequences of the previous two sections, applied to Galois representations.  The section begins, though, with an axiomatic result, Theorem~\ref{axiomaticthm}, on Galois representations and associated Hecke eigenvalues for general reductive groups.
Theorem~\ref{charpcor} is a consequence of Theorem~\ref{axiomaticthm}, as the relevant conditions hold by Theorem~\ref{TandTheta}.  We conclude by showing that some results on iterations of theta (preliminary results on theta cycles) from modular forms naturally extend to Hilbert--Siegel and Hermitian modular forms when the weight is parallel scalar, without restriction on the degree of the base field.

\subsection{Acknowledgements}
The third and fifth named authors thank Dick Gross, Arun Ram, Olav Richter, and Martin Weissman for helpful suggestions.  The first named author thanks Caltech for hospitality during a visit to work on this project.
We are very grateful to the referee for their careful reading and many useful suggestions, which helped improve the paper.

\section{Background and conventions}\label{prelim_sec}

This section recalls key background information about Galois representations (Section~\ref{AG-bkgd}) and associated automorphic forms over Shimura varieties (Section~\ref{AF-bkgd}).

\subsection{Hecke algebras and Galois representations}\label{AG-bkgd}

We recall properties of algebraic groups, Hecke algebras, and Galois representations associated with systems of Hecke eigenvalues. 
{For clarity,  here and in Section~\ref{AT-sec} we work with a general connected reductive group $G$, although in the rest of the paper we deal exclusively with $G$  either symplectic or unitary (as in Section~\ref{AF-bkgd}).}

Our main reference is~\cite{Gross-satake}, with some aspects borrowed from~\cite{treumannvenkatesh} and~\cite{buzzgee14}.

\subsubsection{Dual group}
Let $G$ be a connected reductive algebraic group over a field $F$.
\textbf{We assume that $G$ is split over $F$}, that is there exists a maximal torus of $G$ that is isomorphic over $F$ to a product of copies of $\gm$.

We fix a split maximal torus $T$ contained in a Borel subgroup $B$ of $G$, all defined over $F$.
The Weyl group of $T$ is $W=N_G(T)/T$.
Let $\xub=\xub(T)=\Hom\left(T_{\Fbar},\gm\right)$ be the character lattice and $\xlb=\xlb(T)=\Hom\left(\gm, T_{\Fbar}\right)$ be the cocharacter lattice.
There is a natural pairing
\begin{equation*}
  \langle,\rangle\colon \xub\times\xlb\to\Hom(\gm,\gm)\cong\ZZ
\end{equation*}
given by $\langle\chi,\mu\rangle=\chi\circ\mu$.

Given a representation of $G$, its restriction to $T$ is a direct sum of characters (the \emph{weights} of the representation).
The \emph{roots} of $G$ are the nontrivial weights of the adjoint representation $\mathrm{Ad}\colon G \rightarrow \Aut(\Lie(G))$.
If one instead considers the adjoint action on $\Lie(B)$, one obtains the \emph{positive roots}.
The set $\dub$ of simple roots consists of those positive roots that cannot be written as the sum of other positive roots.
The coroot basis $\dlb$ is the set of simple coroots.
The \emph{root datum} of $G$ is the quadruple $\Psi(G)=(\xub,\dub,\xlb,\dlb)$.

We consider $\hat{G}$, the dual group of $G$, viewed as a reductive group over $k$, a field to be specified later.
(In the automorphic literature, $k$ is often taken to be $\CC$, but our main focus will be on $\fpbar$.
For the existence of $\hat{G}$ over a general base, see~\cite[Theorem 6.1.16]{conrad-reductive}.)
After fixing $\hat{T}\subset\hat{B}\subset\hat{G}$, there is an identification $\xub(\hat{T})\cong \xlb(T)$ that maps the positive roots with respect to $\hat{B}$ to the positive coroots with respect to $B$.
The root datum of $\hat{G}$ is ${\Psi(G)}^\vee=(\xlb,\dlb,\xub,\dub)$.

Under the duality between $G$ and $\hat{G}$, a character $\chi\colon T\to\gm$ in $\xub(T)$ corresponds to a cocharacter $\hat{\chi}\colon\gm\to\hat{T}$ in $\xlb(\hat{T})=\xub(T)$.
Similarly, the character of $\hat{T}$ corresponding to $\mu\in\xlb(T)$ is denoted $\hat{\mu}$.
Note that, given $\chi\in\xub(T)$ and $\mu\in\xlb(T)$, the two morphisms
\[\chi\circ\mu\colon\gm\to T\to\gm \quad\text{and}\quad  \hat{\mu}\circ\hat{\chi}\colon\gm\to\hat{T}\to\gm \]
agree, i.e. $\chi\circ\mu=\hat{\mu}\circ\hat{\chi}$.

Let
\begin{equation*}
  P^+ := \{\chi\in\xlb\mid\langle\alpha,\chi\rangle\geq 0\text{ for all }\alpha\in\dub\}.
\end{equation*}
The elements of $P^+$ are called 
the \emph{dominant weights} of (the maximal torus $\hat{T}$ of) $\hat{G}$. 
There is a partial ordering $\geq$ on $P^+$, where $\lambda\geq\mu$ if 
\begin{equation*}
  \lambda-\mu=\sum_{\alpha^\vee\in\dlb} n_{\alpha^\vee} \alpha^\vee
  \quad\text{with }n_{\alpha^\vee}\in\ZZ_{\geq 0}.
\end{equation*}

By Chevalley's theorem~{\cite[Corollary 2.7 of Part II]{Jantzen}}, 
the set $P^+$ is in bijection with the set of irreducible finite dimensional $\hat{G}$-modules over $k$. For any $\lambda\in P^+$, we denote by $(\rho_\lambda,V_\lambda)$ the irreducible $\hat{G}$-module of highest weight $\lambda$.

Writing $R(\hat G)$ for the representation ring of $\hat G$, we have an identification~\cite[Section 1]{Gross-satake}
\begin{equation*}
  R(\hat{G}) \cong {\ZZ[X^\bullet(\hat{T})]}^W \cong {\ZZ[X_\bullet(T)]}^W,
\end{equation*}
where $W$ acts on $X_\bullet(T)$ by postcomposition and on $X^\bullet(\hat T)$ by the identification $\xub(\hat{T})\cong \xlb(T)$ above.
Using this we can view the elements of the representation ring as sums of characters of $\hat{T}$ or cocharacters of $T$.
Given $\lambda \in P^+$,  let $\chi_\lambda := \Tr(\rho_\lambda|_{\hat T}) \in \ZZ[\xub(\hat{T})]$ denote the character of the irreducible representation $\rho_\lambda\colon \hat{G}\to\GL(V_\lambda)$ of highest weight $\lambda$.
Using the above, we can express
\begin{equation*}
  \label{eq:char_sum}
  \chi_\lambda = \sum_{\mu \leq \lambda} (\dim V_\lambda(\mu)) \hat{\mu},
\end{equation*}
where $\hat{\mu} \in\xub(\hat{T})$ is the character corresponding to $\mu \in\xlb(T)$ by duality, and $V_\lambda(\mu)$ is the weight space of weight $\mu$ in $V_\lambda$.
For further discussion see the remark after~{\cite[Corollary 2.7 of Part II]{Jantzen}}.

\subsubsection{Local Hecke algebra and Satake isomorphism}\label{Hecke-Algs-prelim}
We now take $G$ to be a connected split reductive group over a local field $F_v$ with ring of integers $\oh_v$.
We choose a uniformizer $\pi_v$ of $\oh_v$, and we let $q_v$ denote the cardinality of the residue field $\oh_v/\pi_v\oh_v$.

As we assume $G$ to be split, there exists a group scheme $\mathcal{G}$ over $\oh_v$ whose generic fiber is $G$ and whose special fiber is reductive.
We set $G_v=\mathcal{G}(F_v)$ and $K_v=\mathcal{G}(\oh_v)$, then $K_v$ is a hyperspecial maximal compact subgroup of the locally compact group $G_v$.

Given a commutative ring $R$, the $R$-valued Hecke algebra of the pair $(G_v, K_v)$ is
\begin{equation*}
  \mathcal{H}(G_v, K_v; R)
  =\left\{h\colon K_v\backslash G_v/K_v\longto R \left|\!\begin{array}{c}h\text{ locally constant,}\\
      \text{compactly supported}\end{array}\!\right. \right\}
\end{equation*}
with multiplication given by
\begin{equation*}
  (h_1 * h_2) (K_v g K_v)
  =\sum_{xK_v\in G_v/K_v} h_1(K_v x K_v) h_2(K_v x^{-1}g K_v).
\end{equation*}
We work with the basis of $\mathcal{H}(G_v, K_v; R)$ consisting of the characteristic functions
\begin{equation*}
  c_\lambda = \mathrm{char}\left(K_v\lambda(\pi_v)K_v\right)\qquad
  \text{for }\lambda\in P^+.
\end{equation*}

The Satake transform is a ring isomorphism
\begin{equation*}
  \mathcal{S}_v\colon \mathcal{H}\left(G_v,K_v;\zqv\right)
  \longto {\zqv\left[\xub(\hat{T})\right]}^W=R\left(\hat{G}\right)\otimes\zqv.
\end{equation*}
If $\lambda\in P^+$ then the image of the basis element $c_\lambda$ can be written\footnote{The element $\rho$ is the half-sum of the positive roots of $G$, but we will not need to know this, only that it is the same in all the identities related to the Satake isomorphism $\mathcal{S}_v$.}
\begin{equation}
\label{eq:satake}
  \mathcal{S}_v(c_\lambda)=
  \sum_{\mu\leq\lambda} b_\lambda(\mu)q_v^{\langle\rho,\mu\rangle}\chi_\mu,
\end{equation}
where $\mu$ runs over the elements in $P^+$ such that $\mu\leq\lambda$, $b_\lambda(\mu)\in\ZZ$ and $b_\lambda(\lambda)=1$.

There is a similar identity expressing $\chi_\lambda$ in terms of images of characteristic functions:
\begin{equation*}
  \chi_\lambda = q_v^{-\langle \rho,\lambda\rangle}
  \sum_{\mu\leq\lambda} d_\lambda(\mu)\mathcal{S}_v(c_\mu),
\end{equation*}
where $d_\lambda(\mu)\in\ZZ$ and $d_\lambda(\lambda)=1$.
For future use, let us record a consequence of this identity:
\begin{equation}
  \label{eq:satake_inv}
  \mathcal{S}_v^{-1}(\chi_\lambda) = q_v^{-\langle \rho,\lambda\rangle}
  \sum_{\mu\leq\lambda} d_\lambda(\mu)c_\mu.
\end{equation}

In most of the paper we work with $\fpbar$-coefficients, for a fixed prime $p$.
If $v$ does not divide $p$, then, after making a choice of square root\footnote{For the subtleties of this choice of square root, look for local and global pseudoroots in~\cite[Section 7]{treumannvenkatesh}.} of $q_v$ in $\fpbar$, we get a mod $p$ version of the Satake transform by tensoring with $\fpbar$:
\begin{equation*}
  \mathcal{S}_v\colon\mathcal{H}(G_v,K_v;\fpbar)\longto {\fpbar[\xub(\hat{T})]}^W=R(\hat{G})\otimes \fpbar.
\end{equation*}
Equation~\eqref{eq:satake_inv} continues to hold in this setting, with $d_\lambda(\mu)\in\fp$ and $d_\lambda(\lambda)=1$.

\subsubsection{Galois representations associated to automorphic forms mod $p$}\label{Galois-bkgd}

We summarize the (conjectural) correspondence between Galois representations and automorphic representations following~\cite{buzzgee14}.

Let $G$ be a connected reductive group over $\QQ$ and let $F / \QQ$ be the extension over which $G$ splits.
Henceforth we let $G$ refer to the base change $G \times_{\Spec~\QQ} \Spec~F$.
Fix a prime $p$.
Given a level, i.e.\ an open compact subgroup $K=\prod_v K_v \subset G(\AA_f)$, we say that a finite place $v$ of $F$ is \emph{bad} (with respect to $p$ and $K$) if $v$ lies above $p$ or if $K_v$ is \emph{not} a hyperspecial maximal compact subgroup of $G_v=G(F_v)$.
Otherwise we say that $v$ is \emph{good}, in which case we are in the situation described in Section~\ref{Hecke-Algs-prelim}.
All but finitely many places are good with respect to $p$ and $K$.

Let $f$ be a Hecke eigenform mod $p$ of level $K$ on $G$, that is a class in the cohomology with $\fpbar$-coefficients of the locally symmetric space defined by $G$ and $K$ (see~\cite[Section 5]{treumannvenkatesh} for details).
(NB:\ The examples of interest to this paper are the automorphic forms coming from Shimura varieties, as defined in Section~\ref{sect:automorphic}.)

For every good place $v$, $f$ defines a character of the local Hecke algebra $\mathcal{H}(G_v, K_v;\fpbar)$, namely the \emph{Hecke eigensystem}
\begin{equation*}
  \Psi_{f,v}\colon \mathcal{H}(G_v,K_v;\fpbar)\to \fpbar
\end{equation*}
which takes an operator $T$ to its eigenvalue:
\begin{equation*}
  Tf = \Psi_{f,v}(T)f.
\end{equation*}
Using the Satake isomorphism $\mathcal{S}_v$ from Section~\ref{Hecke-Algs-prelim}, we can define a character $\omega_f\colon R(\hat{G})\otimes \fpbar\to \fpbar$ by
\begin{equation*}
  \omega_f(\chi_\lambda)=\Psi_{f,v}\left(\mathcal{S}_v^{-1}(\chi_\lambda)\right).
\end{equation*}
The characters of the representation ring $R(\hat{G}) \otimes \fpbar$ are indexed by the semi-simple conjugacy classes in $\hat{G}(\fpbar)$.
Given such a class $s$, the corresponding character $\omega_s$ is determined by
\begin{equation*}
  \omega_s(\chi_\lambda)=\chi_\lambda(s).
\end{equation*}
In particular, the character $\omega_f$ corresponding to $f$ is indexed by the conjugacy class $s_{f,v}$ of some semisimple element of $\hat{G}(\fpbar)$; we call $s_{f,v}$ the \emph{$v$-Satake parameter} of $f$.

We are ready to state the expected relation between Hecke eigenforms and Galois representations:
\begin{conj}[Positive characteristic form of{~\cite[Conjecture 5.17]{buzzgee14}}]\label{conj:galois}
  Let $f$ be a mod $p$ Hecke eigenform of level $K$ on a connected reductive group $G$ over $\QQ$, split over a number field $F$.
  There exists a continuous representation 
  \begin{equation*}
    \rho\colon \gf\longto\hat{G}(\fpbar)
  \end{equation*}
  that is unramified outside the finite set $\Sigma$ of places that are bad with respect to $p$ and the level 
  $K$.
  If $v\notin\Sigma$ is a finite place of $F$, then $\rho(\Frob_v)=s_{f,v}$, the $v$-Satake parameter of $f$.
\end{conj}

As noted in~\cite[Remark 5.19]{buzzgee14}, the representation $\rho$ need not be unique up to conjugation.
Given a Hecke eigenform $f$, we will denote by $R(f)$ the set of Galois representations attached to $f$ as in Conjecture~\ref{conj:galois}.

\begin{rmk}
The reader looking to reconcile the statement of Conjecture~\ref{conj:galois} with that of~\cite[Conjecture 5.17]{buzzgee14} will note the following differences:
\begin{itemize}
      \item We work with Galois representations mod $p$ rather than $p$-adic.
      \item Our assumption that $G$ is split over $F$ means that the Galois representations land in the dual group $\hat{G}$ rather than the $L$-group $\L{G}$.
      \item We relate the Hecke eigensystem of $f$ with the images of Frobenii by using Satake parameters (as defined above) rather than by comparing representations of Weil groups.
      \item Since Buzzard and Gee work in the more general setting of automorphic representations, they need to impose the condition of $L$-algebraicity, which is automatically satisfied in the setting of our automorphic forms.
    \end{itemize}
\end{rmk}

For the current state of the art regarding Conjecture~\ref{conj:galois}, we have great achievements by Scholze in~\cite{scholze} for the case $G = \GL_n$.
(The conjectural correspondence in this setting was first formulated by Ash in~\cite{ash-galois}.)
Scholze's work gives the following.
\begin{thm}[{\cite[Theorem I.3]{scholze}}]\label{modpL}
  If $F$ is a CM field that contains an imaginary quadratic field (or more generally, assuming the statements in~\cite{arthur}, if $F$ is totally real or CM), then for any system of Hecke eigenvalues occurring in the singular cohomology $H^i\left(X_K, \fp\right)$ of the locally symmetric space $X_K$ attached to $\GL_n$ over $F$, there is a continuous semisimple representation $\gf\to\GL_n(\fpbar)$ whose Frobenius eigenvalues agree with that system of Hecke eigenvalues.
\end{thm}
The method involves recognising the cohomology of $X_K$ within the cohomology of a Shimura variety attached to a symplectic or unitary group, which are our primary cases of interest.
For results directly in those settings, see~\cite{kretshin} and~\cite{chenevierharris}, respectively.

\subsection{Shimura varieties and automorphic forms}\label{AF-bkgd}

This section introduces the Shimura varieties and spaces of automorphic forms with which we will work.

\subsubsection{Shimura varieties}
In this paper, we consider PEL-type Shimura varieties of either unitary (A) or symplectic (C) type.  Here,
we briefly introduce the Shimura datum $\left(D, *, V, \langle,\rangle, h\right)$ of PEL-type needed for our work later in the paper.  
We refer to~\cite{kottwitz} for a more detailed treatment.  

Let $D$ be a finite-dimensional simple $\QQ$-algebra, and let $F$ be its center.
Let $*$ be a positive involution on $D$ over $\QQ$, and let $F_0$ be the fixed field of $*$ on $F$.
Since $*$ is positive on  $F$, its fixed field $F_0$ is totally real.  We say that $*$ is of the first kind if $F=F_0$, and of the second kind if $F$ is a quadratic imaginary extension of $F_0$ (in which case $F$ is a CM field).  

In the following, we denote by $\TT_{F_0}$ (resp. $\TT_F$) the set of embeddings $\tau\colon F_0\to \RR$ (resp. $\tau\colon F\to \CC$).  If $*$ is of the second kind (case (A)), then for each $\tau\in\TT_{F_0}$, we fix an extension $F\to \CC$ in $\TT_F$, which by abuse of notation, we still denote by $\tau$,  and we write $\tau^*$ for the other embedding $F\to \CC$ that restricts to $\tau$ on $F_0$.  We shall also fix a choice $\Sigma_F$ of CM type, i.e.\ a set consisting of exactly one of $\tau, \tau^\ast$ for each $\tau\in \TT_F$.

 Let $V$ be a non-zero finitely-generated left $D$-module, and let $\langle ,\rangle$ be a non-degenerate $\QQ$-valued alternating form on $V$ such that $\langle bv,w\rangle=\langle v,b^*w\rangle$, for all $v,w\in V$ and all $b\in D$.
Let $C$ be the $\QQ$-algebra ${\rm End}_D(V)$; it is a simple algebra with center $F$, and has an involution $*$ coming from the form $\langle,\rangle$.
Let $h\colon\CC\to C_\RR$ be a $*$-homomorphism such that the symmetric real-valued bilinear form $\langle \cdot,h(i)\cdot\rangle$ on $V_\RR$ is positive-definite.

Associated with the above data, we define the algebraic group $G$ over $\QQ$ whose points on a $\QQ$-algebra $R$ are given by 
\[G(R):=\{x\in C\otimes_\QQ R\mid xx^*\in R^\times\}.\]
We denote by $\nu\colon G\to \gm$ the morphism  $x\to xx^*$, called the \emph{similitude factor}, by $G_1$ its kernel, and by $\hat{\nu}\colon\gm\to \hat{G}$ the cocharacter of $\hat{G}$ corresponding to $\nu$ under duality.  (Note that for any integer $m\geq 1$, the cocharacter of $\hat{G}$ dual to $\nu^m$ is $\hat{\nu}^m$.)
Then $G_1$ is obtained from an algebraic group $G_0$ over $F_0$ by restriction of scalars from $F_0$ to $\QQ$. If $*$ is of the second kind---case (A)---then $G_0$ is an inner form of a quasi-split unitary group over $F_0$. If $*$ is of the first kind, then over an algebraic closure of $F_0$, the group $G_0$ is either orthogonal---case (D)---or symplectic---case (C). 
\textbf{Going forward, we assume we are in case (A) or (C).}

The endomorphism $h_\CC=h\times_\RR \CC$ of $V_\CC=V_\RR\otimes \CC=V\otimes_\QQ \CC$ gives rise to a decomposition  $V_\CC=V_1\oplus V_2$, where for all $z\in \CC$, 
$(h(z),1)=h(z)\times 1$
acts by $z$ on $V_1$ and by $\bar{z}$ on $V_2$. The \emph{reflex field} $E$ of the Shimura datum $\left(D, *, V, \langle,\rangle, h\right)$ is the field of definition of the $G(\CC)$-conjugacy class of $V_1$.

Let $p$ be a rational prime. We write $\ZZ_{(p)}$ for the localization of $\ZZ$ at $p$.  We choose our data so that the following conditions are met:
\begin{enumerate}
\item The prime $p$ is unramified in $F$.
\item The algebra $D$ is split at $p$, i.e. $D_{\QQ_p}$ is a product of matrix algebras over (unramified) extensions of $\QQ_p$.
\item There exists a $\ZZ_{(p)}$-order $\CO_D$ in $D$ that is preserved by $*$ and whose $p$-adic completion is a maximal order in $D_{\QQ_p}$.
\item There exists a $\ZZ_p$-lattice $L$ in $V_{\QQ_p}$ that is self-dual for $\langle,\rangle$ and preserved by $\CO_D$.  
\end{enumerate}

We also specify a \emph{level} $K$, an open compact subgroup of $G(\adeles_f)$, where $\adeles_f$ denotes the ring of finite adeles of $\QQ$. 
We further assume that $K$ is neat (as defined in~\cite[Definition 1.4.1.8]{la}) and that it decomposes as $K=K^p K_p$, where $K^p\subset G\left(\adeles_f^{(p)}\right)$ and $K_p\subset G(\QQ_p)$ is hyperspecial, namely $K_p$ is the stabilizer of $L$ in $V_{\QQ_p}$.

We define $\Sh=\Sh_K$ to be the PEL-type moduli space of abelian varieties of level $K$, associated with the datum 
$(D,*,V,\langle,\rangle,h)$.
Under our assumptions, $\Sh_K$ extends canonically to a smooth quasi-projective scheme over $\CO_E\otimes \ZZ_{(p)}$, which by abuse of notation we still denote by $\Sh$.   
The set $\Sh(\C)$  is a (finite) disjoint union of Shimura varieties obtained from the data $(G,h, K)$. In the following, by abuse of language, we refer to $\Sh$ as the PEL-type Shimura variety of level $K$. 

Given an abelian variety $A$,
we denote by $A^t$ its dual (and, likewise, in the context of sheaves, we also use a superscript $t$ to denote the dual).

\subsubsection{Simplifying conditions}
\textbf{We assume $p$ is totally split in the reflex field $E$.} Then, by~\cite{wedhorn}, the ordinary locus of the modulo $p$ reduction $\Shp$ of the Shimura variety $\Sh$ is non-empty. We fix a prime $\p$ of $E$ above $p$, and we regard $\Sh=\Sh_K$ as a scheme over $\WittZp=\ZZ_p=\CO_{E,\p}$.

For convenience, we further \textbf{assume that $D=F$ and that $p$ splits completely in $F$}. These two assumptions are not necessary, but they allow us to simplify notation. 
(By assumption (2) above, Morita equivalence reduces the general case of $D\not=F$ to that of $D=F$; we refer to~\cite[Section 6]{EiMa} for an explanation of how the case of $p$ totally split in $E$, but not in $F$, relates to the case of $p$ totally split in $F$.)

\subsubsection{Signatures and Levi subgroups}\label{signandlevi_sec}

The decomposition $F\otimes_\QQ\CC=\oplus_{\tau\in\TT_{F}} \CC$
induces a decomposition $V_{1}=\bigoplus_{\tau\in \TT_{F}} V_{1,\tau}$. The \emph{signature} of the Shimura datum is the collection of integers  for all $\tau\in\TT_{F}$
\[a_\tau:=\dim_\CC V_{1,\tau}.\]
Let $2n= \dim_{F_0} V$. In the unitary case (A), we have $a_\tau\oplus a_{\tau^*}=n$, for all  $\tau\in\TT_{F}$, and the tuple of pairs $(a_\tau,a_{\tau^*})$, for $\tau\in\Sigma_F$, is the signature of the unitary group $G_0/F_0$. 
In the symplectic case (C), we have $a_\tau=n$, for all  $\tau\in\TT_{F_0}$.
(For $F_0=\QQ$,  we refer to case (C)  as the \emph{Siegel case}, and to case (A)  as the \emph{Hermitian case} if $a_\tau=a_\tau^*$.)

Let  $\overline{\QQ}_p$ be an algebraic closure of $\QQ_p$, we denote by $\CC_p$ its $p$-adic completion.
We fix an isomorphism $\iota:\CC\to\CC_p$,  compatible
with the  choice of a prime $\mathfrak{p}$ of $E$ above $p$. The map $\tau\mapsto\iota\circ \tau$ defines a bijection between the complex embeddings $\tau:F\to\CC$ and the $p$-adic embeddings $F\to \overline{\QQ}_p$. By assumption, $p$ splits completely in $F$, hence we deduce an identification of $\TT_F$ with the set of primes of $F$ above $p$ (resp. $\TT_{F_0}$ with the set of primes of $F_0$ above $p$), and write $\mathcal{O}_F\otimes_{\ZZ} \ZZ_p=\oplus _{\tau\in\TT_{F}} \ZZ_p$.

Starting from the signature of the Shimura datum, we define an algebraic group $H$ over $\ZZ$ as
\[
  H:=\prod_{\tau\in \TT_F} \GL_{a_\tau}
  =\begin{cases}
  \displaystyle\prod_{\tau\in\TT_{F_0} }\left( \GL_{a_\tau}\times  \GL_{a_{\tau^*}}\right)\subseteq \prod_{\tau\in\TT_{F_0}} \GL_n & \text{in the unitary case (A);}\\
  \displaystyle\prod_{\tau\in\TT_{F_0}} \GL_n & \text{in the symplectic case (C).}
  \end{cases}
\]

Note that $H_\CC$ can be identified with the Levi subgroup of $G_{1,\CC}$ that preserves the decomposition $V_\CC=V_1\oplus V_2$.
  Via $\iota$, we also identify $H_{\ZZ_p}$ with a Levi subgroup of $G_{1,\ZZ_p}$. (Here, with abuse of notation, we still denote by $G$ the reductive model over $\ZZ_p=\mathcal{O}_{E,\mathfrak{p}}$ associated with the  $\ZZ_p$-lattice $L$ in $V_{\QQ_p}$.
Note that all groups are split over $\ZZ_p$). 

We denote by $T\subseteq B$ the diagonal maximal torus inside the Borel subgroup of upper triangular matrices of $H$, and by $N$ the unipotent subgroup of $B$.

\subsubsection{Weights and Schur functors}
We recall relevant aspects of the  theory of algebraic representations of general linear groups, adapted to the group $H=\prod_{\tau\in \TT_F} {\rm GL}_{a_\tau}$.
We refer to~\cite[Section 2.4]{EFMV} for a more detailed discussion.

Over an algebraically closed field of characteristic $0$, in particular over $\overline{\QQ}_p$, the irreducible algebraic representations of $H$  (up to isomorphism) are in one-to-one correspondence with the dominant weights of  $T$. For each dominant weight $\kappa$, let $(\rho_\kappa, V_\kappa)$ denote the irreducible algebraic representation of $H$ over $\overline{\QQ}_p$ of highest weight $\kappa$.  Via Schur functors,  we construct $\ZZ_p$-models of the representations $(\rho_\kappa, V_\kappa)$.

We identify the set of dominant weights of (the maximal torus $T$ of) $H$ with the set
\[{X(T)}_+:=\Big\{(\kappa_{1,\tau},\ldots, \kappa_{a_\tau,\tau})\in\prod_{\tau\in\TT_{F}}\Z^{a_\tau}\Big|\kappa_{i,\tau}\geq \kappa_{i+1,\tau} \text{ for all } i\Big\},\]
via the morphism $\prod_{\tau\in\TT_{F}}\diag(t_{1,\tau},\dots, t_{a_\tau,\tau})\mapsto \prod_{\tau\in\TT_{F}}\prod_{1\leq i\leq a_\tau} t_{i,\tau}^{\kappa_{i,\tau}}$.
For $\kappa\in {X(T)}_+$,  $\kappa=(\kappa_\tau)$, we write
\begin{align*}
|\kappa|:=\sum_{\tau\in\TT_F}|\kappa_\tau| \text{ and } |\kappa_\tau|=\sum_{1\leq i\leq a_\tau} \kappa_{i,\tau}.
\end{align*}
We call a dominant weight $\kappa\in {X(T)}_+$ \emph{positive} if $\kappa\not=0$ and $\kappa_{a_\tau,\tau}\geq 0$ for all $\tau\in\TT_F$.

Let $\kappa$ be positive dominant weight (using  twists by appropriate powers of the determinant, all constructions and definitions below extend to non-positive dominant weights).
Let $y_\kappa$ be the associated {\em Young 
symmetrizer}. More precisely,  we define $y_\kappa=\otimes_{\tau\in\TT_F} y_{\kappa_\tau}$, where for each $\tau\in\TT_F$, $y_{\kappa_\tau}\in \ZZ[\mathfrak{S}_{|\kappa_\tau|}]$ is the Young symmetrizer associated with the dominant weight $\kappa_\tau$ of $H_\tau={\rm GL}_{a_\tau}$,  in the group algebra of the symmetric group $\mathfrak{S}_{|\kappa_\tau|}$ on $|\kappa_\tau|$ letters.

The $\kappa_\tau$-\emph{Schur functor} on the category of $R$-modules, for $R$ any ring\footnote{Schur functors are also defined on the category of locally free sheaves on a scheme $S$.}, is defined by
\[
\mathbb{S}_{\kappa_\tau} (M)=M^{\otimes |\kappa_\tau|} \cdot y_{\kappa_\tau} \subset M^{\otimes |\kappa_\tau|} .
\]
When $R$ is a $\ZZ_p$-algebra, or an algebraically closed field of characteristic $0$, the module $\mathbb{S}_{\kappa_\tau }(R^{a_\tau})$ is an irreducible representation of $H_\tau$ over $R$, with highest weight $\kappa_\tau$.

For each positive weight $\kappa\in {X(T)}_+$,  we define the $\kappa$-Schur functor on the category of $R$-modules $M$ endowed with a decomposition $M=\oplus_{\tau\in\TT_F} M_\tau$,
by
\[ 
\mathbb{S}_{\kappa} (M)=\boxtimes_{\tau\in\TT_F} \mathbb{S}_{\kappa_\tau} (M_\tau)=\left(\otimes_{\tau\in\TT_F} M^{\otimes |\kappa_\tau|}\right)\cdot y_\kappa \subset M^{\otimes |\kappa|}
\]

Let $W=\oplus_{\tau\in\TT_F} \ZZ_p^{a_\tau}$, be the standard representation of $H$ over $\ZZ_p$.
Then the module $\mathbb{S}_{\kappa }(W)$ is a $\ZZ_{p}$-model of the irreducible representation of $H$ with highest weight $\kappa$.
In the following, with abuse of notation and language, for each dominant weight $\kappa$, we still denote  by 
$(\rho_\kappa,V_\kappa)$ the  representation $\mathbb{S}_{\kappa }(W)$, and refer to it as the irreducible representation of highest weight $\kappa$.

\begin{remark}\label{tensor}
By construction, the Young symmetrizer defines an epimorphism
\[ y_\kappa: W^{\otimes|\kappa|}\to \mathbb{S}_{\kappa }(W) .\]
By~\cite[Lemma 2.4.6]{EFMV}, if $\kappa,\lambda$ are two positive dominant weights, then $|\kappa+\lambda|=|\kappa|+|\lambda|$ and there exists a projection map $\pi_{\kappa,\lambda}:\mathbb{S}_{\kappa}(W)\otimes\mathbb{S}_{\lambda}(W)\to\mathbb{S}_{\kappa+\lambda}(W)$ such that the morphism
$y_{\kappa+\lambda}: W^{\otimes|\kappa+\lambda|}\to \mathbb{S}_{\kappa +\lambda }(W)$ factors via the morphism
\[y_\kappa\otimes y_\lambda: W^{|\kappa+\lambda|}= W^{\otimes|\kappa|} \otimes W^{|\lambda|} \to \mathbb{S}_{\kappa }(W)\otimes \mathbb{S}_{\lambda }(W),\]
in other words $y_{\kappa+\lambda}=\pi_{\kappa,\lambda}\circ (y_{\kappa}\otimes y_{\lambda})$.

It is important to note that, as detailed in~\cite[Section 2.4]{EFMV}, all these maps are defined over $\ZZ_{p}$.
\end{remark}

\begin{remark}
  The Schur functors and the various projection maps are an important ingredient of our construction of the differential operators on $p$-adic automorphic forms in Section~\ref{sect:diff-padic}.
  For the construction of the differential operators on mod $p$ automorphic forms in Section~\ref{sect:diff-modp}, we use\footnote{There are other approaches, for instance via the characteristic-free theory of Schur functors~\cite{ABW} and the corresponding Littlewood--Richardson rule~\cite{Boffi}, but using reduction modulo $p$ from the $\ZZ_{p}$-models we defined is sufficient for our purposes.} the reductions modulo $p$ of the Schur functors and projection maps defined above over $\ZZ_{p}$.
  As indicated in Remark~\ref{p>n}, while this provides us with a definition of the necessary maps in characteristic $p$, some of these maps may be zero when $p$ is small.
\end{remark}

\subsubsection{Admissible weights} We now introduce the set of dominant weights by which Maass--Shimura differential operators (as well as the analogous differential operators studied in Section~\ref{analyticcontinuation-section}) can raise the weight of an automorphic form. 

We call a dominant weight $\kappa\in {X(T)}_+$ \emph{even} if $\kappa_{i,\tau}\equiv 0\mod 2$ for all $\tau\in \TT_F$ and all $1\leq i\leq a_\tau$, and we call $\kappa$ \emph{sum-symmetric} if
$\sum_{1\leq i\leq a_\tau}\kappa_{i,\tau}=\sum_{1\leq i\leq a_{\tau^*}}\kappa_{i,\tau^*}$ for all $\tau\in \TT_F$.

Let $W=\oplus_{\tau\in\TT_F} W_\tau$, $W_\tau=\ZZ_p^{a_\tau}$, be the standard representation of $H$ over $\ZZ_p$.

\begin{defi}\label{admissible-defi}
A dominant weight $\lambda$ is called \emph{admissible of depth $e_\lambda=e$}, for $e\in \Z_{>0}$, if the corresponding irreducible algebraic representation of $H$ occurs as a constituent of the representation
${(W^2)}^{\otimes e}$ for
\[W^2:=\begin{cases}
\bigoplus_{\tau\in\TT_{F_0}}  {\rm Sym}^2 W_\tau & \text{in the symplectic case,}\\
  \bigoplus_{\tau\in\TT_{F_0}} (W_\tau\otimes W_{\tau^\ast}) & \text{in the unitary case.}
\end{cases}\]
\end{defi}
By definition, admissible weights are positive. Furthermore,  a dominant weight $\lambda$ is \emph{admissible} in the symplectic case if and only if it is positive and even; and
in the unitary case, a dominant weight $\lambda$ is \emph{admissible} 
if and only if it is positive and sum-symmetric.

Note that if $\lambda$ is admissible, then $|\lambda |$ is positive and even, and the Young symmetrizer \[\yo_\lambda\colon W^{\otimes |\lambda|}\to W_\lambda\] induces an epimorphism ${(W^{ 2})}^{\otimes |\lambda |/2}\to W_\lambda$, which by abuse of notation, we still denote by $\yo_\lambda$. Hence, in particular, $\lambda$ is admissible of depth $e_\lambda=|\lambda|/2$.

\begin{rmk}

We denote scalar weights by ${(\underline{k_\tau})}_{\tau\in \TT_F}:={(k_\tau,\dots, k_\tau)}_{\tau\in\TT_F}$ for integers $k_\tau$.  Furthermore, if there exists an  integer $k$ such that $k_\tau = k$ for all $\tau$, then we write $\underline{k}$ for ${(\underline{k_\tau})}_{\tau\in \TT_F}$.
In the symplectic case, a scalar weight ${(\underline{\varla_\tau})}_{\tau\in\TT_F}$ is admissible if and only if $\varla_\tau\geq 0$ and  even, for all $\tau\in\TT_F$. Hence, in the Siegel case,  the admissible scalar weights are exactly the positive multiples of $\underline{2}$.
In the unitary case, a scalar weight ${(\underline{\varla_\tau})}_{\tau\in\TT_F}$ is admissible if and only if $a_\tau\varla_\tau=a_{\tau^*}\varla_{\tau^*}$, for all $\tau\in\TT_F$.
Hence, in the Hermitian case, a scalar weight is admissible if and only if $\varla_{\tau}=\varla_{\tau^*}$,  for all $\tau\in\TT_F$.

\end{rmk}

\subsubsection{Automorphic forms}\label{sect:automorphic} We recall the construction of automorphic sheaves over $\Sh$. We refer to~\cite[Section 3.2]{CEFMV} for details.

Given a characteristic $p$ field $\mathbb{F}$, we denote its Witt vectors by $W(\mathbb{F})$.  We also set $\WittZp:= W(\ZZ/p\ZZ)$ and identify $\WittZp$ with $\ZZ_p$. Following the conventions of Section~\ref{signandlevi_sec}, we identify $\TT_{F_0}$ (resp. $\TT_{F}$) with ${\rm Hom}\left(\CO_{F_0},\Witt\right)$ (resp.  ${\rm Hom}\left(\CO_{F},\Witt\right)$).

Let $\pi\colon \Auniv\to \Sh$ denote the universal abelian scheme, of relative dimension $g=n[F_0\colon\QQ]$ over $\Sh$, 
and set $\omega:=\pi_*\Omega^1_{\Auniv/\Sh}$; it is a locally free sheaf of rank $g$.

Set $\TT:=\TT_F$.  The action of  $\CO_F$ on $\omega$ induces a decomposition of $\omega$ as 
\[\omega=\bigoplus_{\tau\in\TT}\omega_\tau\]
where for each $\tau\in\TT$ the sheaf $\omega_\tau$ is locally free of rank $a_\tau$.

For any dominant weight $\kappa$ of $T$,  the \emph{automorphic sheaf} $\omega^\kappa$ of weight $\kappa$   over $\Sh$
is defined as follows. Let ${\mathcal E}/\Sh$  be the sheaf
\[{\mathcal E}:=\bigoplus_{\tau\in\TT}{\underline{\rm Isom}}_{\CO_\Sh}\left(\CO_\Sh^{a_\tau},\omega_{\tau}\right);\]
it admits a natural left action of $H$.  Given an irreducible representation $(\rho_\kappa,V_\kappa)$  of $H$ of highest weight $\kappa$, we define $\omega^\kappa:={\mathcal E}\times^{\rho_\kappa} V_\kappa$ so that for each $\CO_{E,\p}$-algebra $R$, $\omega^\kappa(R):=(\CE\times V_\kappa\otimes R)/\left((\ell, m)\equiv (g\ell, \rho_\kappa({}^t g^{-1})m)\right)$.

An \emph{automorphic form} of weight $\kappa$ and level $K$, defined over an $\CO_{E,\p}$-algebra $R$, is a global section of the sheaf $\omega^\kappa$ on $\Sh_K\times_{\CO_{E,\p}} R$.

\subsubsection{Hasse invariants and $p$-adic automorphic forms}\label{hasse_sec}

\newcommand{\ha}{h}

We recall the construction and properties of Hasse invariants relevant to our settings.  Details most pertinent to the present paper are available, for various contexts, in~\cite{GorenHasse, conrad-hasse, GoldringNicole} and~\cite[Section 7]{AndreattaGoren}.   (More general constructions have been given more recently, e.g.~\cite{KWhasse}, but such generality is unnecessary for the present paper.)

Let $\pi\colon \Auniv\to \Shp$ denote the universal abelian scheme over the mod $p$ reduction $\Shp$ 
 of the Shimura variety $\Sh$.
Setting  $\omega_{\Auniv/\Shp}:=\pi_*\Omega^1_{\Auniv/\Shp}$, we have the Hodge filtration  of $\hdrp$ over $\Shp$:
\[0\to\omega_{\Auniv/\Shp}\to \hdrp\to \rpi_\Auniv\to 0.\]

Let $\absfrob\colon\Shp\to \Shp$ denote the absolute Frobenius on $\Shp$;
we denote by
$\Auniv^{(p)}:=\Auniv\times_{\Shp,\absfrob} \Shp$ the pullback of $\Auniv$ under $\absfrob$, and by
$\relfrob\colon\Auniv\to \Auniv^{(p)}$ the relative Frobenius of $\Auniv$.
The morphism $\relfrob$ induces an $\cO_\Shp$-linear map \[\relfrob^*\colon\rpi_{\Auniv^{(p)}}=\rpi_\Auniv^{(p)} \to \rpi_\Auniv;\]
we denote the dual map by 
\begin{align*}
\ha\colon\omega_{\Auniv/\Shp}\to \omega_{\Auniv^{(p)}/ \Shp}=\omega_{\Auniv/\Shp}^{(p)}.
\end{align*}

\begin{defi}\label{def:hasse}
The \emph{Hasse invariant} is the automorphic form \[\hasse:=\det \ha\in H^0\left(\Shp, \vert \omega_{\Auniv/\Shp}\vert ^{p-1}\right)\]
(where $|\cdot|$ denotes the top exterior power).
\end{defi}

By construction, the ordinary locus $\shpord$ of $\Shp$ agrees with the complement of the vanishing locus of the Hasse invariant $\hasse$.

Recall that a sufficiently large power of $\hasse$ is known to lift to characteristic zero. 
For any $m\geq 1$, set $\Sh_m:=\Sh\times_{\WittZp} \WittZp/p^m$,
and denote by $\shord_m$ the locus of $\Sh_m$ where (a sufficiently large power of) the Hasse invariant $\hasse$ is invertible. 
Define $\shord:=\varinjlim_{m}\shord_m$ as a formal scheme over $\WittZp$.  The formal scheme $\shord$ is the \emph{formal ordinary locus} over $\WittZp$.

\begin{defi}
  For any dominant weight $\kappa$, a \emph{$p$-adic automorphic form of weight $\kappa$} is a section of
\[H^0(\shord,\omega^\kappa):=\varprojlim_{m} H^0(\shord_m,\omega^\kappa).\]
\end{defi}

\section{Analytic continuation of differential operators}\label{analyticcontinuation-section}\label{analytic_sec}
This section establishes analytic continuation of the mod $p$ reduction of certain $p$-adic differential 
operators $\Diff^\lambda_\kappa$ (analogues of $\ci$ Maass--Shimura differential operators), \emph{a priori} defined over the ordinary locus, to the whole Shimura variety.  (Recall from the previous section that we assume we are in case (A) or (C), i.e.\ the unitary or symplectic case.)  This section's main result is Theorem~\ref{ana_thm}, of which Theorem~\ref{continuationB} is a consequence.  Via different methods, de Shalit and Goren obtained related results on analytic continuation, when one restricts to quadratic imaginary fields $\cmfield$ and scalar weights~\cite{DSG2}.  As discussed in Section~\ref{intro-newresults}, unlike that approach, the techniques below handle all signatures and all weights simultaneously.  The techniques below also do not rely on explicit computations via $q$-expansions (and consequently handle fields $\cmfield$ of any degree and automorphic forms of any weight).

\subsection{Key ingredients for the construction of our differential operators}\label{ingreds-diffops}
Let $T$ be a scheme, and let $Y$ be a smooth scheme over $T$. 
Suppose $\pi\colon\mathscr{A}\to Y$ is a polarized abelian scheme. In particular, $\pi$ is a smooth and proper morphism.
Let $\omega:=\omega_{\mathscr{A}/Y}:=\pi_*\Omega^1_{\mathscr{A}/Y}\subseteq \hdrAY$, and consider the Hodge filtration
\[0\to\omega\hookrightarrow \hdrAY\to \rpi_\mathscr{A}\to 0.\]

Our differential operators are built from the Gauss--Manin connection
\[\nabla=\nabla_{\mathscr{A}/Y}\colon \hdrAY\to\hdrAY\otimes\Omega^1_{Y/T}\]
and the Kodaira--Spencer morphism 
\[\KS=\KS_{\mathscr{A}/Y}\colon\omega\otimes\omega\to\Omega^1_{Y/T}.\]
By definition, $KS_{\mathscr{A}/Y}:=\langle \cdot, \nabla (\cdot)\rangle_{\mathscr{A}}$, where $\langle \cdot, \cdot\rangle_{\mathscr{A}}$ is the pairing induced by the polarization on $\mathscr{A}$ and extended linearly in the second variable to a pairing between $\omega$ and $\nabla(\omega)$ (and using the fact that $\omega$ is an isotropic subspace of $\hdrAY$ under this pairing).
The Kodaira--Spencer morphism induces an isomorphism
\[\ks\colon\omega^2\isomto \Omega^1_{Y/T},\]
where 
\begin{align*}
\omega^2:=
\begin{cases}
\bigoplus_{\tau\in \TT_{F_0}}{\rm Sym}_{\cO_Y}^2\left(\omega_\tau\right) & \mbox{in the symplectic case}\\
\bigoplus_{\tau\in \TT_{F_0}}\left(\omega_\tau\otimes_{\cO_Y}\omega_{\tau^*}\right) & \mbox{in the unitary case}.
\end{cases}
\end{align*}
For details on the Gauss--Manin connection and the Kodaira--Spencer (iso)morphism, we refer the reader to~\cite[Sections 2.1.7 and 2.3.5]{la} and~\cite[Section 9]{FaltingsChai}.

\subsection{Differential operators on $p$-adic automorphic forms}\label{sect:diff-padic}
We briefly recall a construction of $p$-adic differential operators $\Diff_\kappa^\lambda$, analogues of $\ci$ Maass--Shimura operators, introduced in~\cite[Chapter II]{kaCM} for Hilbert modular forms and extended to the Siegel and unitary cases in~\cite{padiffops2, padiffops1}  and~\cite{EDiffOps, EFMV}, respectively.

Let $\shord$ denote the formal ordinary locus over $\witt$.  Write $\omega=\omega_{\Auniv/\shord}$. 
Define \[U\subseteq \hdrord\] to be Dwork's \emph{unit root} subcrystal\footnote{Under the identification of $\hdrord$ with the Dieudonn\'e crystal of $\Auniv[p^\infty]/\shord$, the subcrystal $U$ is the maximal \'etale subcrystal of $\hdrord$.} (introduced in~\cite{ka73}).
Then $U$ is a complement of $\omega$, and the inclusion $\omega\subseteq \hdrord$ composed with the projection 
$\pr_U:\hdrord\to\hdrord/U$
  is an isomorphism. We denote the induced morphism  by 
\[\PR_U: \hdrord\to\omega_{\Auniv/\shord}.\]

Define
\[\Diff:=(\PR_U\otimes \ks^{-1})\circ\nabla\vert _{\omega}: \omega \subseteq \hdrord\to\hdrord\otimes\Omega^1_{\shord /\witt}\to
 \omega\otimes\omega^2.\]
For any $d\geq 1$, by the Leibniz rule (the product rule), we obtain an operator

\[\Diff_d:  \omega^{\otimes d} \to \omega^{\otimes d}\otimes \omega^2\subseteq \omega^{\otimes(d+2)}.\]
By the construction of Schur functors (see also~\cite[Proposition 3.3.1]{EFMV}), for any dominant weight $\kappa$, the operator $\Diff_d$  for $d=|\kappa|$ induces an operator
\[\Diff_\kappa:\omega^\kappa\to \omega^\kappa \otimes \omega^2.\]
For any admissible weight $\lambda$, the corresponding weight-raising differential operators 
\[\Diff^\lambda_\kappa:\omega^\kappa\to \omega^{\kappa+\lambda}\]
are induced via Schur functors by the $e$-th iterations $\Diff^{(e)}_d:=\Diff_{d+2(e-1)}\circ \cdots\circ\Diff_{d+2}\circ\Diff_{d}$, 
\[\Diff^{(e)}_d: \omega^{\otimes d} \rightarrow \omega^{\otimes d}\otimes {(\omega^2)}^{\otimes e}\subseteq\omega^{\otimes d+2e},\]
composed with the Young symmetrizer $\yo_\lambda: {(\omega^2)}^{\otimes e}\to\omega^\lambda$, for $e=|\lambda|/2$ (see Remark~\ref{tensor}).

\begin{rmk}
For any admissible weight $\lambda$ and any $p$-adic form $f$ of weight $\kappa$, $\Diff_\kappa^\lambda(f)$ is a $p$-adic form of weight $\kappa+\lambda$. 
In particular, for any integer $m\geq 1$ and any mod $p^m$ automorphic form $f\in H^0(\Sh_{m}, \omega^\kappa)$, we obtain  $\Diff_\kappa^\lambda (f)\in H^0(\shord_m, \omega^{\kappa+\lambda})$. By construction of the Hasse invariant $\hasse$, for a sufficiently large integer $N>>0$, the section $\hasse^N \cdot\Diff_\kappa^\lambda (f) \in H^0\left(\shord_m, \omega^{\kappa+\lambda+{\underline{(p-1)N}}}\right)$ extends (uniquely) to all of $\Sh_{m}$, i.e.\ $\hasse^N \cdot\Diff_\kappa^\lambda (f)$ is, in fact, the restriction of an element of $H^0\left(\Sh_{m}, \omega^{\kappa+\lambda+\underline{(p-1)N}}\right)$.

We shall prove in Theorem~\ref{ana_thm} that when $m=1$, for each admissible weight $\lambda$, there exists an explicit integer, namely $N=|\lambda|/2$, such that for all dominant weights $\kappa$ and all characteristic $p$ automorphic forms $f$ of weight $\kappa$, the sections
 $\hasse^N\cdot \Diff_\kappa^\lambda (f)$, defined \emph{a priori} over the ordinary locus, extend to all of $\Shp:=\Sh_1$. 
 \end{rmk}

\begin{rmk}\label{p>n}
Recall that $2n:=\dim_{F_0}V$ as defined in Section~\ref{signandlevi_sec}.  By~\cite[Section 5.2, especially the proof of Proposition 5.2.4]{EFMV},
when $p>n$,  the weight-raising differential operators $\Diff^\lambda$ modulo $p$, for $\lambda$ symmetric, exist and are nonzero.  In fact, by \emph{loc.\ cit.}, in the unitary case, it suffices to work with $p>\max_{\tau\in\TT}(\min(a_\tau, a_{\tau^*}))$.
\end{rmk} 

\subsection{Differential operators on  automorphic forms modulo $p$}\label{diffmop_sec}
Let $\Shp$ denote the reduction modulo $p$ of the Shimura variety $\Sh$, i.e. $\Shp:=\Sh_1$, and set $\omega=\omega_{\Auniv/\Shp}$.  We now construct a new class of weight-raising differential operators $\Theta_\kappa^\lambda$ on the space $H^0(\Sh_{1}, \omega_{\Auniv/\Shp}^\kappa)$ of automorphic forms in characteristic $p$.  
  For the special case $G=\GSp_4$, Yamauchi constructed and studied similar operators in~\cite{Yama}.

\subsubsection{Adjugates and the Hasse invariant}
For a morphism $f\colon \cF\to\cG$  of locally free $\cO_X$-modules of rank $g$, the \emph{adjugate} of $f$ is the morphism 
\[\adj{f}\colon\cG\otimes \vert\cF\vert\to\cF\otimes\vert\cG\vert\]
(recall $|\cdot|$ denotes the top exterior power) obtained from the dual map 
\[{(\wedge^{g-1} f)}^t\colon \left({\mathcal G}\otimes \vert {\mathcal G}\vert^{-1}\right) \simeq {\left(\wedge^{g-1}\cG\right)}^t
\to {\left(\wedge^{g-1}{\mathcal F}\right)}^t\simeq \cF \otimes \vert {\mathcal F}\vert^{-1}\]
after tensoring with the identity map on $\vert \cF\vert \otimes\vert\cG\vert$. 
It satisfies the property that 
\begin{align}\label{adj}
  \adj{f}\circ (f\otimes \id_{\vert\cF\vert})=\id_\cF\otimes\det(f)\colon \cF\otimes   \vert {\mathcal F}\vert\to \cF\otimes   \vert {\mathcal G}\vert.
\end{align}

  Let $\relfrob\colon \Auniv\rightarrow \Auniv^{(p)}$ denote the relative Frobenius morphism on the universal abelian scheme over $\Shp$.  Denote by $\ha\colon\omega_{\Auniv/\Shp}\to \omega_{\Auniv/ \Shp}^{(p)} $ the dual of the morphism
$\relfrob^*\colon \rpi_{\Auniv}^{(p)}\to \rpi_\Auniv$ introduced in Section~\ref{hasse_sec}.
In the following, by abuse of notation, we still denote by \[\adj{\ha}\colon\omega_{\Auniv/\Shp}^{(p)}\to \omega_{\Auniv/ \Shp}\otimes \vert \omega_{\Auniv/\Shp}\vert ^{p-1} \] the map obtained by tensoring the adjugate of $\ha$  
with the identity on $\vert \omega_{\Auniv/\Shp}\vert^{-1}$ and then
composing with the identification $\vert \omega_{\Auniv/\Shp}^{(p)}\vert\simeq \vert \omega_{\Auniv/\Shp}\vert^p$. 
From equality~\eqref{adj}, we deduce that \[\adj{\ha}\circ \ha\colon\omega_{\Auniv/\Shp} \to \omega_{\Auniv/\Shp}\otimes \vert \omega_{\Auniv/\Shp}\vert^{p-1}\] is equal
to multiplication by the Hasse invariant $\hasse=\det (\ha)$, as introduced in Definition~\ref{def:hasse}.

Define\footnote{More precisely, $\curlyU$ is the sheafification of the image of $\relfrob^*$.}
\begin{align*}
\curlyU := \Image\left(\relfrob^*\colon \hdrp^{(p)}\rightarrow \hdrp\right).
\end{align*}

\begin{lem}\label{adjugateha}
 The relative Frobenius  $\relfrob\colon\Auniv\to\Auniv^{(p)}$ induces an isomorphism
 \begin{align}\label{isomodcurlyU}
 \hdrp/\curlyU\simeq \omega
^{(p)}.
\end{align}
Under this identification, the inclusion $\omega\subseteq \hdrp$ composed with the projection 
\begin{align*}
\pr_{\curlyU}\colon\hdrp\to\hdrp/\curlyU
\end{align*}
 agrees with the morphism $\ha\colon\omega\to \omega^{(p)}$.
\end{lem}
\begin{proof} By Katz's work on the conjugate Hodge--de Rham spectral sequence~\cite[Section 2.3]{Katz-diff}  (also~\cite[Corollary 1.6., \S{}5.1, Proposition 5.1]{wedhorn09}, and~\cite[\S{}5.1]{BBM}), the sheaf $\curlyU$ and  the quotient $\hdrp/\curlyU$ are locally free of rank $g$ and dual to each other under the pairing $\langle ,\rangle_{\Auniv}$ 
on $\hdrp$ induced by the polarization $\mu_\Auniv$ of $\Auniv$.
Furthermore, it follows from this work of Katz that the relative Frobenius
\begin{align*}
\relfrob^*\colon\hdrp^{(p)}\to\hdrp
\end{align*}
 induces an isomorphism 
\begin{align}\label{katzcurlyUiso}
\rpi_{\Auniv}^{(p)}\simeq\curlyU.
\end{align}
Dualizing Isomorphism~\eqref{katzcurlyUiso}, we obtain Isomorphism~\eqref{isomodcurlyU}, as well as that the inclusion $\omega\subseteq \hdrp$ composed with the projection $\pr_{\curlyU}\colon\hdrp\to\hdrp/\curlyU$ agrees with the morphism $\ha\colon\omega_{\Auniv/\Shp}\to \omega_{\Auniv/\Shp}^{(p)}$ dual to $\relfrob^*\colon \rpi_{\Auniv}^{(p)}\to \rpi_\Auniv$.
\end{proof}

\subsubsection{A construction of differential operators in characteristic $p$}\label{sect:diff-modp}
We continue to set $\omega=\omega_{\Auniv/\Shp}$.
Consider the morphism
\[\PR_{\curlyU}:= \adj{h}\circ \pr_\curlyU\colon \hdrp \to\hdrp/\curlyU\simeq \omega^{(p)}\to\omega\otimes\vert\omega\vert^{p-1}.\]
Define
\[\Theta:=(\PR_\curlyU\otimes \ks^{-1})\circ\nabla\vert_{\omega}\colon \omega \to\hdrp\otimes\Omega^1_{X /\witt}\to
\left( \omega\otimes\vert\omega\vert^{p-1}\right)\otimes \omega^2= \omega\otimes\left(\vert\omega\vert^{p-1}\otimes \omega^2\right).\]

More generally, we construct an operator
\[\Theta_d\colon  \omega^{\otimes d} \to \omega^{\otimes d}\otimes \left(\vert\omega\vert^{p-1}\otimes \omega^2\right)\subseteq \omega^{\otimes d+p+1}\]
as follows.
Like in~\cite[Section 3.1]{EiMa2} (or similarly,~\cite[Section 3.1.2]{EDiffOps} and~\cite[Section 4.3]{hasv}), we use the product rule to extend the Gauss--Manin connection to a connection
\begin{align*}
\nabla_{\otimes d}\colon {\left(\hdrp\right)}^{\otimes d}\to{\left(\hdrp\right)}^{\otimes d}\otimes\Omega^1_{X/\witt},
\end{align*}
via
\begin{align*}
\nabla_{\otimes d}\left(f_1\otimes\cdots\otimes f_d\right) = \sum_{j=1}^d \iota_j\left(f_1\otimes\cdots\otimes\nabla\left(f_j\right)\otimes\cdots\otimes f_d\right),
\end{align*}
where $\iota_j$ is the isomorphism
\begin{align*}
\iota_j: {(\hdrp)}^{\otimes j}\otimes\Omega_{Y/\witt}^1\otimes{(\hdrp)}^{\otimes (d-j)}&\isomto {(\hdrp)}^{\otimes d}\otimes\Omega_{X/\witt}^1\\
e_1\otimes \cdots \otimes e_j\otimes u \otimes e_{j+1}\otimes\cdots\otimes e_d&\mapsto e_1\otimes\cdots\otimes e_d\otimes u.
\end{align*}
So composing with the inclusion $\omega^{\otimes d}\hookrightarrow {\left(\hdrp\right)}^{\otimes d}$, we get a morphism
\begin{align}\label{nabladimg}
\nabla_{\otimes d}: \omega^{\otimes d}\rightarrow \sum_{j=1}^{d}\left(\omega^{\otimes (j-1)}\otimes\hdrp\otimes\omega^{\otimes(d-j)}\otimes \Omega_{X/\witt}^1\right).\
\end{align}
For any $d\geq 1$, we then obtain an operator 
\[\Theta_d:=(\PR_\curlyU\otimes \ks^{-1})\circ\nabla_{\otimes d}\colon  \omega^{\otimes d} \to \omega^{\otimes d}\otimes \left(\vert\omega\vert^{p-1}\otimes \omega^2\right)\subseteq \omega^{\otimes d+p+1}.\]
This is similar to Katz's construction of the theta operator in~\cite[Remark, p. Ka-5 to Ka-6]{Katz-theta} and relies on the fact that the image of $\nabla_{\otimes d}$ restricted to $\omega^{\otimes d}$ is as in~\eqref{nabladimg}.

By the construction of Schur functors, for any dominant weight $\kappa$, the operator $\Theta_d$  for $d=|\kappa|$ induces an operator
\[\Theta_\kappa\colon\omega^\kappa\to \omega^\kappa \otimes \left(\vert\omega\vert^{p-1}\otimes\omega^2\right).\]

For any admissible weight $\lambda$ of depth $e=|\lambda|/2$, the corresponding differential operators 
\[\Theta^\lambda_\kappa\colon\omega^\kappa\to \omega^{\kappa+\lambda+\underline{(p-1)e}}\]
are induced via Schur functors applied to the $e$-th iteration $\Theta^{(e)}_d:=\Theta_{d+(p+1)(e-1)}\circ \cdots\circ\Theta_{d+p+1}\circ\Theta_{d}$,
\[\Theta^{(e)}_d\colon \omega^{\otimes d} \to \omega^{\otimes d}\otimes {\left(\vert\omega\vert^{p-1}\otimes \omega^2\right)}^{\otimes e}\subseteq \omega^{\otimes d+(p+1)e},\]
composed with projection \[\id_{\vert\omega\vert^{(p-1)e}}\otimes \yo_\lambda\colon  {\left(\vert\omega\vert^{p-1}\otimes \omega^2\right)}^{\otimes e}=  \vert\omega\vert^{(p-1)e}\otimes {\left(\omega^{2}\right)}^{\otimes e}
\to \vert\omega\vert^{(p-1)e}\otimes \omega^\lambda=\omega^{\lambda+\underline{(p-1)e}}.\]

\begin{rmk}
In Section~\ref{ana_sec}, we shall be particularly interested in the action of $\Theta_\kappa^\lambda$ on the restriction $\omega_{\Auniv/\Shp}|_{\Shpord}  = \omega_{\Auniv/\Shpord}$ to the ordinary locus $\Shpord:=\shord_1\subset \Shp$.  For convenience of notation, we write $\Theta_\kappa^\lambda|_\Shpord$ to denote the resulting morphism of sheaves $\omega_{\Auniv/\shpord}^\kappa\to \omega_{\Auniv/\shpord}^{\kappa+\lambda+\underline{(p-1)|\lambda|/2}}$ over $\Shpord$.  Similarly, we write ${\Pi_\curlyU}\vert_{\shpord}$ for the morphism $\hdrordp\to \omega_{\Auniv/\shpord} \otimes \vert \omega_{\Auniv/\shpord} \vert^{p-1}$ obtained by restriction to sheaves over $\shpord\subset \Shp$.
\end{rmk}

\subsection{Analytic continuation modulo $p$}\label{ana_sec}

We now prove the operators $\Theta_\kappa^\lambda$ analytically continue the mod $p$ reduction of the $p$-adic differential operators $\Diff_\kappa^\lambda$, \emph{a priori} defined only over the ordinary locus $\Shpord:=\shord_1$, to all of $\Shp:=\Sh_1$.
More precisely, we establish the following result, of which Theorem~\ref{continuationB} from Section~\ref{intro-newresults} is a consequence.
\begin{thm}[Analytic Continuation]\label{ana_thm}
For any admissible weight $\lambda$ and dominant weight $\kappa$,
 \[ {\Theta_\kappa^\lambda}\vert_{ \shpord} \equiv \hasse^{|\lambda |/2} \cdot \Diff_\kappa^\lambda  \mod p \]
 as morphisms $\omega_{\Auniv/\shpord}^\kappa\to \omega_{\Auniv/\shpord}^{\kappa+\lambda+\underline{(p-1)|\lambda|/2}}$ of sheaves over $\shpord$.  
\end{thm}
In particular,  for any classical automorphic form $f$ of weight $\kappa$, the $\bmod p$ reduction of the $p$-adic automorphic form $E^{|\lambda |/2}\cdot\Diff_\kappa^\lambda f$ is classical.

\begin{proof}
Comparing the constructions of the operators $\Theta_\kappa^\lambda$ and $\Diff_\kappa^\lambda$, the statement  reduces to the following lemma.
\end{proof}

\begin{lem} Maintaining the above notation,
\[{\Pi_\curlyU}\vert_{\shpord} \equiv \hasse\cdot \Pi_U\mod p\]
as maps $\hdrordp\to \omega_{\Auniv/\shpord} \otimes \vert \omega_{\Auniv/\shpord} \vert^{p-1}$.
\end{lem}
\begin{proof}
We work over $\shpord$ and write $\omega=\omega_{\Auniv/\shpord}$. Consider the pullback of the slope filtration to $\hdrordp^{(p)}$, namely $0\subseteq U^{(p)}\subset \hdrordp^{(p)}$. By construction of the unit root subcrystal $U$, the restriction of the relative Frobenius map $\relfrob^*$ to $U^{(p)}$ induces an isomorphism onto $U$. Since $\relfrob^*$  is equal to $0$ on $\omega^{(p)}$, it follows from the definition of $\curlyU$ that $\curlyU=U$ over $\shpord$.

By Lemma~\ref{adjugateha}, the induced morphism \[\omega\simeq \hdrord/U\to \hdrord/\curlyU\simeq \omega^{(p)}\] agrees with the restriction to $\shpord$ of the morphism $\ha\colon\omega\to\omega^{(p)}$. We deduce  that  $\pi_\curlyU \equiv h  \circ\Pi_U $ over $\shpord$,
and hence
\[\Pi_\curlyU =\adj{\ha}\circ \pi_\curlyU\equiv \adj{\ha}\circ h  \circ\Pi_U \equiv \hasse\cdot \Pi_U\mod p.\]
\end{proof}

\begin{rmk}
Since the weights of ${\Theta_\kappa^\lambda}\vert_{ \shpord}(f)$ and $\Diff_\kappa^\lambda (f)$ are different, we can only compare the values of ${\Theta_\kappa^\lambda}\vert_{ \shpord}(f)$ and $\Diff_\kappa^\lambda (f)$ at a point of $\shpord$ after identifying modular forms of different weights in a larger, common space of $p$-adic modular forms $V$ (like in~\cite[Section 8.1.3]{hida} or~\cite[Section 4.2.1]{CEFMV}).  For the purpose of the present paper, we need not concern ourselves with such details of comparisons between forms of different weights.  We note, though, that since the $q$-expansions of the Hasse invariant $\hasse$ are identically $1 \mod p$, the $q$-expansions of ${\Theta_\kappa^\lambda}\vert_{ \shpord}(f)$ and $\Diff_\kappa^\lambda (f)$ agree mod $p$, which implies ${\Theta_\kappa^\lambda}\vert_{ \shpord}(f)\equiv\Diff_\kappa^\lambda (f)\mod p$ inside $V$.
\end{rmk}

\begin{rmk}\label{power-remark}

In Theorem~\ref{ana_thm}, we show that in order to analytically continue $\Diff^\lambda$ to the whole Shimura variety, it suffices to multiply $\Diff^\lambda$ by a power of $\hasse$  which depends only by the depth of $\lambda$ (in the sense of Definition~\ref{admissible-defi}).  This guarantees that the resulting operator is defined on automorphic forms on the whole $\mod$ $p$ Shimura variety.  If one specializes to a specific weight $\lambda$ of a given depth, though, it is possible that a smaller power of $\hasse$ suffices to cancel the poles introduced by $\Diff^\lambda$.

For example, this is seen for scalar weights in~\cite[Proposition 3.9]{Yama}.
Working with $\GSp_4(\IQ)$ and toward a different goal from us, Yamauchi shows that $E\Diff^{\underline{2}}(f)$ is entire for any classical form $f$ (by comparison, Theorem~\ref{ana_thm} gives the operator $E^2\Diff^{\underline 2}$).  Although only carried out in~\cite{Yama} for $\GSp_4(\IQ)$, his ideas might be considered more generally, and offer a  possible route to determining a lower power of $\hasse$ by which one needs to multiply the operator $\Diff^{\lambda}$.  In particular, for a scalar weight $\lambda = \underline{2k}$,  with $k\in \ZZ_{>0}$, one would expect the operator $\hasse ^k \Diff^{\underline{2k}}$ to be entire.
In the spirit of Katz's approach from~\cite[Proof of Lemma 3]{Katz-theta} (which covers the case of $\GSp_2(\IQ)$), the key input is a careful analysis of $\nabla$, $\KS$, and Frobenius working locally on the base, with explicit coordinates.
Indeed, Proposition~\ref{RC_prop} illustrates that
in the special case of scalar weight Siegel modular forms over $\IQ$ of sufficiently large weight, we need not have multiplied by such a high power of $\hasse$, as one copy suffices.

\end{rmk}

\subsubsection{Relation to B\"ocherer--Nagaoka's work with Rankin--Cohen brackets}\label{BNRC-section}
In the case of scalar-valued classical Siegel modular forms (when the totally real field is $\IQ$) of level $1$ that arise as the reduction $\mod p$ of classical characteristic $0$ Siegel modular forms,
work of B\"ocherer--Nagaoka~\cite[Theorem 4]{BoechererNagaoka-firsttheta} yields a
different approach to the analytic continuation of the operator $\Diff^{\underline{2}}$ modulo $p$, for $p\geq n+3$.
We denote by $\left[\,,\right]$ the generalized Rankin--Cohen bracket constructed for pairs of Siegel modular forms of weight $\geq n/2$ in~\cite{Ibukiyama} and~\cite{EholzerIbukiyama}.
In~\cite[Section 4]{BoechererNagaoka-firsttheta}, B\"ocherer and Nagaoka define a theta operator (denoted here by $\Theta_{\mathrm{BN}}$) on classical Siegel modular forms $f$ of level $1$ and scalar weight $k$ that arise as the $\mod p$ reduction of classical characteristic $0$ Siegel modular forms.
(As noted in Remark~\ref{rmk:lift}, not all intrinsic $\mod p$ Siegel modular forms are reductions from characteristic zero.)

When $2k\geq n$, B\"ocherer and Nagaoka express $\Theta_{\mathrm{BN}}(f)$  in terms of the normalized pairing $\{\,, \}:= {(2\pi\sqrt{-1})}^{-n}\left[\,,\right]$ (in~\cite[proof of Theorem 4]{BoechererNagaoka-firsttheta}).
When $k<n/2$, $\Theta_{\mathrm{BN}}(f) = 0$,  as noted in~\cite[proof of Theorem 4]{BoechererNagaoka-firsttheta}.  This vanishing follows immediately from the formula for the action of $\Theta_{\mathrm{BN}}$ on $q$-expansions $\sum a(T) q^T$ of mod $p$ reductions of characteristic $0$ (archimedean) Siegel modular forms, combined with the fact that classical characteristic $0$ (archimedean) Siegel modular forms are singular, i.e.\ their Fourier coefficients $a(T)$ vanish at all half-integral positive $T$ (as explained in~\cite[Section 1]{Klingen} and first proved in~\cite{Resnikoff, Freitag}).

\begin{prop}\label{RC_prop}

Assume $p\geq n+3$.  Then, for  any scalar-valued Siegel modular form $f$ (where the totally real field is $\IQ$) of level $1$ and weight  $k$,
\begin{align*}
\Theta^{\underline{2}}(f)=
\begin{cases}
 \frac{{(-1)}^n}{(n+1)}\hasse^{n-1}\left\{f,\hasse\right\}  \bmod p, & \mbox{ if } k\geq n/2\\
 0,  \mbox{ if } k<n/2.
 \end{cases}
 \end{align*}
\end{prop}

\begin{proof}

First we deal with the case $k\geq n/2$, i.e.\ the range where $\{\,, \}$ is defined.
Under the assumptions $p\geq n+3$ and $2k\geq n$, both $\hasse^{n-1}\{ f , \hasse\}$ and $\Theta^{\underline{2}}\equiv\hasse^n\cdot \Diff^{\underline{2}} (f)$  are Siegel modular forms, of scalar weight $k+ 2+n(p-1)$. Hence, by the $q$-expansion principle, it suffices to compare their  $q$-expansions. 
By comparing the action of $\left\{\cdot,\hasse\right\}$ in~\cite[proof of Theorem 4]{BoechererNagaoka-firsttheta} with the action of $\Diff^{\underline{2}}$ on $q$-expansions in~\cite[Theorem 9.2]{EDiffOps}, combined with the formula for the effect of the Schur functors on coefficients (as provided, e.g., in~\cite[Corollary  5.2.10]{EFMV}), we deduce
 \[ \frac{{(-1)}^n}{ (n+1)! } \{\cdot,\hasse\}|_{\shpord} \equiv \frac{1}{n!}\hasse \Diff^{\underline{2}}\mod p.\]
 (Note that $\Theta^{\underline{2}}f(q) = n!\Theta_{\mathrm{BN}}f(q)$.)
Under the assumption $p>n+3$, both constants $\frac{{(-1)}^n}{ (n+1)!}$ and $\frac{1}{n!}$ are $p$-adic units, so their quotient $\frac{{(-1)}^n}{ (n+1)}$ is a $p$-adic unit as well.  So, applying Theorem~\ref{ana_thm}, we have
 \[\Theta^{\underline{2}}\equiv\hasse^n \Diff^{\underline{2}} \equiv \frac{{(-1)}^n}{(n+1)}\hasse^{n-1}\left\{\cdot,\hasse\right\} \mod p.\]

In the case $k<n/2$, as noted above, $f$ is singular.  Similarly to the argument of~\cite{BoechererNagaoka-firsttheta} mentioned above, we immediately see from the action of $\Theta^{\underline{2}}$ on $q$-expansions that $\Theta^{\underline{2}}(f) = 0$ for all singular forms.
\end{proof}

\begin{rmk}\label{rmk:lift}
 As noted in the Remark at the end of~\cite[Section 1]{BoechererNagaoka-firsttheta},
 the B\"ocherer--Nagaoka operator applies to reductions modulo $p$ of scalar-weight classical Siegel modular forms (when the totally real field is $\IQ$) in characteristic zero, whereas our operators apply to intrinsic characteristic $p$ Siegel modular forms.  
 (As explained in Remark~\ref{power-remark}, specializing to a specific weight in our setting can enable lowering our exponent on $\hasse$ to match theirs.)
It is expected that, in low weights, the space of intrinsic forms is strictly bigger than the space of reductions.
Stroh~\cite[Th\'eor\`eme 1.3]{stroh-classicite},~\cite[Th\'eor\`eme 1.1]{stroh-relevement} and Lan--Suh~\cite[Corollary 4.3]{LanSuh} show that in the scalar-valued case, the two spaces coincide when $k \geq n+2$ and $p\geq n(n+1)/2$.

Our focus on the geometry provides results for a broader set of cases, including vector-weights (as well as lower scalar weights, since we do not restrict ourselves to reductions mod $p$ of characteristic $0$ forms),
Siegel--Hilbert associated with totally real fields of arbitrary degree (as opposed to degree $1$ in~\cite{BoechererNagaoka-firsttheta}), and unitary groups of arbitrary signature.
\end{rmk}

\section{Commutation relations with Hecke operators}\label{Hecke_sec}

The goal of this section is to explain how the differential operators $\Theta_\kappa^\lambda$ (resp. $\Diff_\kappa^\lambda$) interact with Hecke operators.  Following the convention of~\cite[Section 3.1]{Edixhoven}, we use the terminology ``commutation relations'' to describe this interaction.
We shall see that the two commute up to  (explicit) scalar multiples. Hence, in particular, we shall deduce that the differential operators $\Theta_\kappa^\lambda$ (resp. $\Diff_\kappa^\lambda$) map Hecke eigenforms to Hecke eigenforms.  
Furthermore, in the case of Hecke operators at $p$, we shall prove that the scalar multiple is a power of $p$, and as an application, deduce the vanishing of the operators $\ee\Diff_\kappa^\lambda=0$, where $\ee$ denotes Hida's ordinary projector. For modular forms, such vanishing was also noted in~\cite[p. 29]{CGJ} and~\cite[p. 345]{Coleman95}.

\subsection{Interaction with isogenies}
Throughout this section, let $Y$ be a scheme that is smooth over a scheme $T$.

\begin{defi}
Given two polarized abelian varieties $(A,\mu_A), (B,\mu_B)$, we say that an isogeny $\phi\colon A\to B$ preserves the polarizations up to a scalar multiple, if there exists $\nu(\phi)\in \QQ^\times$, called the \emph{similitude factor}, such that 
\[\phi^t\circ \mu_B\circ \phi={[\nu(\phi)]}_{A^t}\circ \mu_A.\]
\end{defi}

Let $(A,\mu_A)$ and $(B, \mu_B)$ be polarized abelian schemes of the same relative dimension $g$ over the scheme $Y/T$.
An isogeny $\phi\colon A\to B$ defined over $Y$ induces a morphism $\phi^*\colon H^1_{\rm dR}(B/Y)\to H^1_{\rm dR}(A/Y)$, which preserves the Hodge filtration, that is $\phi^* (\omega_{B/Y})\subseteq \omega_{A/Y}$. By abuse of notation, for each dominant weight $\kappa$,
we still denote by \[\phi^*\colon\omega_{B/Y}^\kappa\to\omega_{A/Y}^\kappa\]
the morphism induced by the map $\phi^*$ via Schur functors.
Note that if the polarizations $\mu_A, \mu_B$ have the same degree, then $\nu\in \ZZ$ and $\deg(\phi)={\nu(\phi)}^g$.
In particular, for $(A,\mu_A)=(B,\mu_B)$, and $\phi={[n]}_A$, then $\nu(\phi)=n^2$; in this case, the morphism ${(\phi^*)}^{\otimes d}$ is multiplication by $n^d$.

\begin{lem}\label{kodairaisogeny}
Let $(A,\mu_A)$ and $(B, \mu_B)$ be polarized abelian schemes of the same relative dimension over $Y/T$, and let $\phi\colon A\to B$ be an isogeny preserving their polarization up to multiplication by a scalar $\nu(\phi)\in\QQ^\times$. Then 
\[\KS_{A/Y}\circ (\phi^*\otimes\phi^*)=\nu(\phi) \KS_{B/Y}.\]
\end{lem}
\begin{proof}
By the definition of  the Kodaira--Spencer morphism
\[\KS_{A/Y}:=\langle \cdot, \nabla_A (\cdot)\rangle_{A},\] 
the statement is equivalent to the equality
\[\langle \phi^*(\cdot), \nabla_A \phi^*(\cdot)\rangle_{A}=\nu(\phi) \langle \cdot, \nabla_B (\cdot)\rangle_{B}.\]
By the functoriality of the Gauss--Manin connection, we deduce  $\nabla_A \circ\phi^*=(\phi^*\otimes \id)\circ\nabla_B$, and reduce the statement to the equality
\[\langle \phi^*(\cdot), \phi^*(\cdot)\rangle_{A}=\nu(\phi) \langle \cdot, \cdot\rangle_{B}.\]
The latter follows from $\phi^t\circ\mu_B\circ \phi=\nu(\phi)\mu_A$.
\end{proof}

\subsection{Commutation relations with prime-to-$p$ Hecke operators}

Following~\cite[Ch.~VII, \S{}3]{FaltingsChai}, we define the action of prime-to-$p$ algebraic correspondences, and of prime-to-$p$ Hecke operators,  on automorphic forms.
For the construction of the moduli spaces of isogenies in the PEL-type case, see~\cite[\S{}1.6--1.7]{BueltelWedhorn}.

Fix a rational prime $\ell\not=p$. \textbf{We assume $\ell$ is good} (see Section~\ref{Galois-bkgd} for definition). 
We denote by
${\mathcal H}_0(G_\ell,\QQ)$ the $\QQ$-subalgebra of the local Hecke algebra $\mathcal{H}(G_\ell,K_\ell;\QQ)$ generated by locally constant functions supported on cosets $K_\ell\gamma K_\ell$, for $\gamma\in G_\ell$ an  integral matrix (see Section~\ref{Hecke-Algs-prelim} for notation).

\begin{defi}
	We denote by $\ell{\rm -Isog}$ the moduli scheme\footnote{In general, $\ell{\rm -Isog}$ is a stack, but it is relatively representable over $\Sh$; as noted in~\cite[\S{}1.7]{BueltelWedhorn}, the latter is a scheme in our setting since we assume that the level is neat.} of $\ell$-isogenies over $\Sh$, and by $\phi\colon {\rm pr}_1^*\Auniv \to {\rm pr}^*_2\Auniv$ the universal $\ell$-isogeny, for \[{\rm pr}=({\rm pr}_1,{\rm pr}_2)\colon \ell{\rm -Isog}\to \Sh\times \Sh\] the natural structure morphism.
\end{defi}

Note that the degree of the universal isogeny is locally constant on $\ell{\rm -Isog}$. For any connected component $Z$ of $\ell{\rm -Isog}$, the two projections ${\rm pr}_i\colon Z\to\Sh$ are proper, and they are finite \'etale over $\Sh[1/\deg(Z,\phi)]$, where $\deg(Z,\phi)$ denotes the degree of $\phi$ on $Z$. In particular, they are finite and \'etale over $\Sh/{\CO_{E,\p}}$.

\begin{defi}
For any dominant weight $\kappa$, there is a natural action of $(Z,\phi)$ on the space $H^0(\Sh,\omega^\kappa)$ of automorphic forms 
  of weight $\kappa$, defined as 
\[T_\phi=T_{(Z,\phi), \kappa}:={\rm tr}\circ \phi^*\circ{\rm pr^*_2}\colon H^0(\Sh,\omega^\kappa)\to H^0(Z,{\rm pr}_2^*\omega^\kappa)\to H^0(Z,{\rm pr}^*_1\omega^\kappa)\to H^0(\Sh,\omega^\kappa)\]
\end{defi}

It follows from the definition that, for any dominant weight $\kappa$,  the operator $T_{(Z,\phi),\kappa}$ induces an action of $(Z,\phi)$ on the space $H^0(\Shp,\omega^\kappa)$ of mod $p$ automorphic forms of weight $\kappa$ (by reduction modulo $p$), and also on the space  $H^0(\shord,\omega^\kappa)$ of $p$-adic automorphic forms of weight $\kappa$ (by restriction).

\begin{defi}\label{defi-heckeaction}

Let $Y$ denote a base scheme such that $\ell$ is invertible in $\cO_Y$, equipped with a map $f\colon Y\to \Sh$ (e.g., $Y=\Sh$ over $\CO_{E,\p}$,  $Y=\Shp$ over ${\bF}_p$, or $Y=\shord$ over $\Z_p$).
 There is a natural $\QQ$-linear map
\[h=h_\ell\colon {\mathcal H}_0(G_\ell,\QQ) \to \QQ[\ell{\rm -Isog}/Y]\]
which to any double coset $K_\ell\gamma K_\ell$, with $\gamma$ an integral matrix in $G_\ell$, associates the union $Z(K_\ell\gamma K_\ell)$ of those connected components of $\ell{\rm -Isog}$ where the universal isogeny is an $\ell$-isogeny of type $K_\ell\gamma K_\ell$.
\end{defi}

By definition, the action on automorphic forms of the prime-to-$p$ Hecke operators agrees with that of the prime-to-$p$ algebraic correspondences (via pullback under $h$).

\begin{thm}\label{TandTheta}  Let $(Z,\phi)$ be a connected component of $\ell{\rm -Isog}$.  
For any dominant weight $\kappa$ and any admissible weight $\lambda$, we have
\begin{enumerate}
\item $T_\phi\circ \Theta^{\lambda}_\kappa  ={\nu(\phi)}^{|\lambda |/2}  \Theta^\lambda_\kappa\circ T_\phi$\label{P1}
\item $T_\phi\circ \Diff^{\lambda}_\kappa ={\nu(\phi)}^{|\lambda |/2} \Diff^\lambda_\kappa\circ T_\phi $\label{P2}
\end{enumerate}
where  $\nu(\phi)$ denotes the similitude factor of $\phi$. 
\end{thm}
\begin{proof}
For Part~\eqref{P1}, by the functoriality of the construction of the operators $\Theta^\lambda_\kappa$, it suffices to establish the equality
\[T_\phi\circ \Theta=\nu(\phi)\Theta\circ T_\phi. \]
By definition, $\Theta:=(\Pi_\curlyU\otimes\ks^{-1})\circ \nabla\vert_{\omega_{\Auniv/\Shp}}$.
By the functoriality of the Gauss--Manin connection $\nabla$, and of the definition of $\Pi_\curlyU$,
it suffices to prove the equality
\[\nu(\phi) \KS=\KS\circ (\phi^*\otimes\phi^*)\colon {\rm pr}_2^*H^1_{\rm dR}(\Auniv/\Shp)\otimes {\rm pr}_2^*H^1_{\rm dR}(\Auniv/\Shp)\to {\rm pr}_1^*\Omega^1_{\Shp/T}.\]
This follows from Lemma~\ref{kodairaisogeny}.

For Part~\eqref{P2}, the statement reduces to the equality  $T_\phi\circ \Diff=\nu(\phi) \Diff\circ T_\phi$, and the same argument applies to the operator $\Diff$, with the morphism $\Pi_U$ in place of $\Pi_\curlyU$.
\end{proof}

\begin{rmk} In the Siegel case, 
for an admissible scalar weight $\lambda=\underline{\varla}$, we have that $\varla$ is even, $|\lambda|=g\varla$, and ${\nu(\phi)}^{|\lambda |/2}={\deg(\phi)}^{\varla/2}$.
In particular,  for $g=2$, Theorem~\ref{TandTheta} specializes to Yamauchi's result~\cite[Proposition 3.9]{Yama}.
\end{rmk}

\subsection{Commutation relations with  Hecke operators at $p$}

Following~\cite[Ch. VII,\S 4]{FaltingsChai}, we define the action of $p$-power algebraic correspondences, and of Hecke operators at $p$, on automorphic forms over the ordinary locus $\shord$.
Following \emph{loc.cit.},  we identify $H\times \gm$ with the appropriate maximal Levi subgroup $M$ of $\G$ over $\ZZ_p$  (see Section~\ref{signandlevi_sec}).  

We write $M_p:=M(\QQ_p)$, and $\gamma\in M_p$ as $\gamma=(\alpha,p^d)$ with $\alpha \in H(\QQ_p)$ and $d\in \ZZ$. Note that 
$\gamma=(\alpha,p^d)\in M_p$ is an integral matrix  if and only if  $d\geq 0$ and both $\alpha$ and $p^d\alpha^{-1}$ are integral matrices in $H(\QQ_p)$.
Also, by definition $M(\ZZ_p)=K_p\cap M_p$, hence the local Hecke algebra $\mathcal{H}(M_p,M(\ZZ_p);\QQ)$ is a subalgebra of $\mathcal{H}(G_p,K_p;\QQ)$  (see Section~\ref{Hecke-Algs-prelim} for notation). 

We denote by
${\mathcal H}_0(M_p,\QQ)$ the $\QQ$-subalgebra of  $\mathcal{H}(M_p, M(\ZZ_p);\QQ)$ generated by locally constant functions supported on cosets $M(\ZZ_p)\gamma M(\ZZ_p)$, for $\gamma\in M_p$ an  integral matrix.
\begin{defi}
We denote by $p{\rm -Isog}^o$ the moduli scheme of $p$-isogenies over the ordinary locus $\shord$, and by $\phi\colon {\rm pr}_1^*\Auniv \to {\rm pr}^*_2\Auniv$ the universal $p$-isogeny, for \[{\rm pr}=({\rm pr}_1,{\rm pr}_2)\colon p{\rm -Isog}^o\to \shord\times \shord\] the natural structure morphism.
\end{defi}
By~\cite[\S{}VII.4, Proposition 4.1]{FaltingsChai}, for any connected component $Z$ of $p{\rm -Isog}^o$, the two projections ${\rm pr}_i\colon Z\to\shord$ are finite and flat over $\shord/\witt$.

\begin{defi}
For any dominant weight $\kappa$, there is a natural action of $(Z,\phi)$ on the space $H^0(\shord,\omega^\kappa)$  of $p$-adic automorphic forms of weight $\kappa$, defined as 
\[T_\phi=T_{(Z,\phi)}:={\rm tr}\circ \phi^*\circ{\rm pr^*_2}\colon H^0(\shord,\omega^\kappa)\to H^0(Z,{\rm pr}_2^*\omega^\kappa)\to H^0(Z,{\rm pr}^*_1\omega^\kappa)\to H^0(\shord,\omega^\kappa).\]
\end{defi}

\begin{thm}\label{TandTheta2} Let $(Z,\phi)$ be a connected component of $p{\rm -Isog}^o$.   
For any dominant weight $\kappa$ and any admissible weight $\lambda$, we have
\[T_\phi\circ \Diff^{\lambda}_\kappa={\nu(\phi)}^{|\lambda |/2} \Diff^\lambda_\kappa\circ T_\phi ,\]
where $\nu(\phi)$ denotes the similitude factor of $\phi$.  
In particular, if $\nu(\phi)>1$, then $T_\phi\circ \Diff^{\lambda}_\kappa \equiv 0 \mod p$.
\end{thm}

\begin{proof} The commutation relations follow by the same argument as in the proof of Part~\eqref{P2} of Theorem~\ref{TandTheta}. The vanishing in positive characteristic follows from the equality $\nu(\phi)=p^d$, for $d\geq 1$.
\end{proof}

Consider the mod $p$-reduction \[{\rm pr}_{\overline{\bF}_p}\colon p{\rm -Isog}^o\otimes_{\Witt} {\overline{\bF}_p}\to \shpord\times \shpord.\]
Note that the degree, similitude factor,  and $p$-type of the universal isogeny are locally constant on $p{\rm -Isog}^o$. (Recall that, by definition,  the $p$-type of a $p$-isogeny with similitude factor $p^d$ is a coset  $H(\ZZ_p)\alpha H(\ZZ_p)$ of an integral matrix $\alpha\in H(\QQ_p)$ such that $p^d\alpha^{-1}$ is also integral.)
In~\cite[\S{}VII.4, Proposition 4.1]{FaltingsChai}, for any connected component $Z$ of $p{\rm -Isog}^o\otimes_{\Witt}\overline{\bF}_p$, Faltings and Chai compute the purely inseparable multiplicity $\mu(Z)$ of the geometric fibers of  ${\rm pr}_i\colon Z\to \shpord$
in terms on the degree and the $p$-type  of the universal isogeny on $Z$. Furthermore, 
they prove
 there is a well-defined integral action of the operator 
${\mu(Z)}^{-1}T_{(Z,\phi)}$ on the space of mod $p$ automorphic forms of weight $\kappa$, for all dominant weights $\kappa$. 

We define  the \emph{normalized} action of $(Z,\phi)$ on mod $p$ automorphic forms  as $t_\phi=t_{(Z,\phi)}:={\mu(Z)}^{-1}T_{(Z,\phi)}$.

\begin{defi}
Let $Y$ denote a base scheme of characteristic $p$, equipped with a map $f\colon Y\to\shpord$.
There is a natural $\QQ$-linear map
\[h=h_p\colon
{\mathcal H}_0(M_p,\QQ) \to \QQ[p-{\rm Isog}^o/Y]\]
which to any double coset $M(\ZZ_p)\gamma M(\ZZ_p)$ with $\gamma=(\alpha, p^d)$ an integral matrix in $M_p$, 
  associates ${\mu(Z(\alpha,d))}^{-1} \cdot Z(\alpha, d)$, where $Z(\alpha ,d)$ denotes the union of those connected components of $p-{\rm Isog}^o$ where the universal isogeny is a $p$-isogeny of $p$-type $H(\ZZ_p)\alpha H(\ZZ_p)$ and similitude factor $p^d$.
\end{defi}

By definition, the action of the Hecke operators at $p$ on mod $p$ automorphic forms agrees with the \emph{normalized} action of $p$-power algebraic correspondences.

\subsubsection{Hida's ordinary projector}
In~\cite[\S{}8]{hida} (also~\cite[\S{}3.6]{H02}), Hida establishes
the divisibility of the Hecke operators at $p$ on $p$-adic automorphic forms by a given power of $p$, keeping their integrality. 
More precisely, he proves that the normalized action on mod $p$ automorphic forms lifts to $p$-adic automorphic forms.

For $1\leq j\leq n$, let $\alpha_j\in H(\QQ_p)$ defined as $\alpha_{j,\tau}:= {\rm diag}[{\mathbb I}_{n-j},p{\mathbb I_j}]$,  $\tau\in\T_0$.  (For any positive integer $k$, $\mathbb{I}_k$ denotes the $k\times k$ identity matrix.) Let  $Z_{\alpha_j}$ denote
the union of those connected components of $p{\rm -Isog}^o$ where the universal isogeny has  $p$-type $H(\ZZ_p)\alpha_j H(\ZZ_p)$ and similitude factor $p$.
For $j=1, \dots, n$, Hida proves the integrality of the operators
${{\mu(\alpha_j)}^{-1}} {\mathbb U}(\alpha_j):={\mu(Z_{\alpha_j})}^{-1}T_{(Z_{\alpha_j},\phi)}$ on the space of $p$-adic automorphic
forms over the ordinary locus.

By definition (\cite[\S{}8.3.1, Lemma 8.12]{hida}),
the ordinary projector $\ee$ on the space of $p$-adic automorphic forms is
\[\ee:=\varinjlim_m {{\mathbb U}(p)}^{m!}\]
where  
${\mathbb U}(p):=\prod_1^n {{\mu(\alpha_j)}^{-1}} {\mathbb U}(\alpha_j) $.

Hence, by Theorem~\ref{TandTheta2}, we deduce the following result.
\begin{coro}\label{ed0}
 For any dominant weight $\kappa$, and admissible weight $\lambda$, we have
\[\ee\Diff_\kappa^\lambda=0.\]
\end{coro}

\section{Applications to Galois representations}\label{galois-section}

Theorem~\ref{TandTheta} implies that the mod $p$ differential operators from  Section~\ref{analytic_sec} map Hecke eigenforms to Hecke eigenforms.
When there are Galois representations arising from the corresponding systems of Hecke eigenvalues, it is natural to ask how our differential operators affect the Galois representations.

\subsection{An axiomatic theorem}\label{AT-sec}
We start with a result on Galois representations attached to systems of Hecke eigenvalues for a general algebraic reductive group $G$ defined over $\QQ$ (and split over a number field $F$), as introduced in Section~\ref{Galois-bkgd}.    

We fix a prime $p$ and write $\chi$ for the mod $p$
cyclotomic character.
Given a mod $p$ Hecke eigenform $f$, we denote $\Sigma_f$ the set of places that are bad with respect to $p$ and the level of $f$, and by $R(f)$ the set of Galois representations $\rho\colon\gf\to\hat{G}(\fpbar)$ that are (conjecturally) attached to $f$ as in Conjecture~\ref{conj:galois}.

\begin{thm}[Axiomatic Theorem]\label{axiomaticthm}
  Let $G$ be a connected reductive group over $\QQ$, split over the number field $F$.
  For $i=1,2$, let $f_i$ be a mod $p$ Hecke eigenform on $G$. For $v\notin\Sigma_i:=\Sigma_{f_i}$, let $\Psi_{i, v}\colon \mathcal{H}(G_v,K_v;\fpbar)\to \fpbar$ denote the Hecke eigensystem of $f_i$ at $v$.
Then
  \begin{equation}
    \label{eq:psi}
    \Psi_{2,v}(c_\lambda) = \eta(\lambda(\pi_v)) \Psi_{1,v} (c_\lambda)
  \qquad\text{for all }\lambda\in P^+,v\notin\Sigma_1\cup\Sigma_2,
  \end{equation}
for a character $\eta$ of $G$
 if and only if  
  \begin{equation}
    \label{eq:main}
    (\hat{\eta}\circ\chi)\otimes\rho_1 \in R(f_2) \qquad\text{for all } \rho_1\in R(f_1),
  \end{equation}
for the cocharacter  $\hat{\eta}$ of $\hat{G}$ corresponding to $\eta$ by duality.
\end{thm}

\begin{rmk}
For Galois representations arising from Hecke eigenforms on symplectic and unitary groups, Theorem~\ref{axiomaticthm}  specializes to Theorem~\ref{charpcor} from Section~\ref{intro-newresults}, as we shall observe in Section~\ref{final-effects-section}.
\end{rmk}

\begin{rmk}
  The use of the tensor product sign in $(\hat{\eta}\circ\chi)\otimes\rho_1$ is generally an abuse of notation and should be understood as in Lemma~\ref{lem:notensor}.
  We retain the $\otimes$ notation because it is evocative of the case when $\hat{G}$ is a group of matrices, where we are indeed dealing with a tensor product of representations.
\end{rmk}

The proof of Theorem~\ref{axiomaticthm} uses the following lemmas.

\begin{lem}\label{lem:etahat}
  Let $G$ be a reductive group and let $\eta\colon G\to\gm$ be a character of $G$.
  Then the cocharacter $\hat{\eta}\colon \gm\to\hat{G}$ has image in the center $Z(\hat{G})$ and
  \begin{equation*}
    \eta\circ\mu=\eta\circ\lambda
  \end{equation*}
  for any dominant weights $\mu,\lambda\in P^+$ such that $\mu\leq\lambda$.
\end{lem}
\begin{proof}
The image of $\eta$ is abelian.  So $\eta(G^\prime)=1$, where $G^\prime$ denotes the derived subgroup of $G$.
  So $\eta$ induces a character of the abelianization of $G$, hence the dual cocharacter $\hat{\eta}$ lands in the center of $\hat{G}$.
Moreover,  since $\eta(G^\prime)=1$, and the coroots of $G$ are the same as the coroots of $G^\prime$, we see that $\eta\circ\alpha^\vee=1$ for any coroot $\alpha^\vee$.
If $\mu$ and $\lambda$ are comparable, their differ by a linear combination of coroots, hence  $\eta\circ\mu=\eta\circ\lambda$.
\end{proof}

\begin{lem}\label{lem:notensor}
  Let $G$ be a reductive group over a field $k$ and let $\eta\colon G\to\gm$ be a character of $G$.
  Let $\rho\colon \Gamma\to \hat{G}(k)$ be a representation of a group $\Gamma$ into the $k$-valued points of the dual group $\hat{G}$.
  Let $\chi\colon \Gamma\to \gm(k)$ be a $k$-valued character of $\Gamma$.
  Then the map $(\hat{\eta}\circ\chi)\otimes \rho\colon\Gamma\to\hat{G}(k)$ defined by
  \begin{equation}
    \label{eq:rho}
    ((\hat{\eta}\circ\chi)\otimes \rho)(\gamma):=\hat{\eta}(\chi(\gamma))\rho(\gamma)
  \end{equation}
  is a representation of $\Gamma$.
\end{lem}
\begin{proof}
  In order to show that Equation~\eqref{eq:rho} defines a group homomorphism, it suffices to prove that $\hat{\eta}(\chi(\gamma))$ is in the center of $\hat{G}(k)$, which was done in Lemma~\ref{lem:etahat}.
\end{proof}

We are ready to prove the main result of this section.

\begin{proof}[Proof of Theorem~\ref{axiomaticthm}]
Let $\rho_1\in R(f_1)$ and define $\rho_2:=(\hat{\eta}\circ\chi) \otimes \rho_1$. Then $\rho_1$ and $\rho_2$  are both unramified at primes $v\notin\Sigma:=\Sigma_1\cup\Sigma_2$, and 
$\rho_2(\Frob_v)=\hat{\eta}(\pi_v)\rho_1(\Frob_v)$ for all $v\notin\Sigma$.

For $v\notin\Sigma$, let $s_{ i,v}$ denote the $v$-Satake parameter of $f_i$,  for  $i=1,2$. 
Then $\rho_i\in R(f_i)$ if and only if  
$s_{i,v}= \rho_{i}(\Frob_v)$ for all $v\notin\Sigma$.
Hence 
Equation~\eqref{eq:main}  is equivalent to the equalities 
\begin{align*}
    s_{2,v}=\hat{\eta}(\pi_v)s_{1,v} \text{ for all }v\notin\Sigma.
  \end{align*}

Fix $v\notin\Sigma$.  Recalling the Satake parameters from Section~\ref{Hecke-Algs-prelim}, note that the equality $s_{2,v}=\hat{\eta}(\pi_v)s_{1,v} $  is equivalent to 
  \begin{align*}
    \chi_\lambda(s_{2,v})=\chi_\lambda(\hat{\eta}(\pi_v)s_{1,v})
    \qquad\text{for all }\lambda\in P^+.
  \end{align*}
  
Let $\lambda\in P^+$, and let $\rho_\lambda\colon\hat{G}(\fpbar)\to\GL(V_\lambda)$ be the irreducible representation of $\hat{G}(\fpbar)$ of highest weight $\lambda$. 
  As a function on the maximal torus $\hat{T}$,  the character of $\rho_\lambda$ can be written as
  \begin{equation*}
    \chi_\lambda=\sum_{\mu\leq\lambda} (\dim V_\lambda(\mu)) \hat{\mu}.
  \end{equation*}
  In particular,
  \begin{align*}
    \chi_\lambda(\hat{\eta}(\pi_v) s_{f_1,v}) &= \sum_{\mu\leq\lambda} (\dim V_\lambda(\mu)) \hat{\mu}(\hat{\eta}(\pi_v) s_{f_1,v})
                                   = \sum_{\mu\leq\lambda} (\dim V_\lambda(\mu)) \hat{\mu}(\hat{\eta}(\pi_v)) \hat{\mu}(s_{f_1,v})\\
                                   &= \sum_{\mu\leq\lambda} (\dim V_\lambda(\mu)) \eta(\mu(\pi_v)) \hat{\mu}(s_{f_1,v}) 
                                   = \eta(\lambda(\pi_v))  \sum_{\mu\leq\lambda} (\dim V_\lambda(\mu)) \hat{\mu}(s_{f_1,v}) \\
                                   &= \eta(\lambda(\pi_v)) \chi_\lambda(s_{f_1,v}),
  \end{align*}
  where we used that $\eta(\mu(\pi_v))=\eta(\lambda(\pi_v))$ whenever $\mu\leq\lambda$ (see Lemma~\ref{lem:etahat}).

Hence, Equation~\eqref{eq:main} holds if and only if 
  \begin{equation}
    \label{eq:chi}
    \chi_\lambda(s_{2,v})=\eta(\lambda(\pi_v)) \chi_\lambda(s_{1,v}) 
    \qquad\text{for all }\lambda\in P^+, v\notin \Sigma.
  \end{equation}

  We use Equation~\eqref{eq:satake_inv}  to show that Equation~\eqref{eq:psi} implies Equation~\eqref{eq:chi} (and hence Equation~\eqref{eq:main}).  So
  \begin{align*}
    \chi_\lambda(s_{2,v}) &= \omega_{2}(\chi_\lambda)
                                       = \Psi_{2,v}\left(\mathcal{S}_v^{-1}(\chi_\lambda)\right)
                                       = \Psi_{2,v}\left(q_v^{\langle-\rho,\lambda\rangle}\sum_{\mu\leq\lambda}d_\lambda(\mu)c_\mu\right)\\
                                       &= q_v^{\langle-\rho,\lambda\rangle}\sum_{\mu\leq\lambda}d_\lambda(\mu)\Psi_{2,v}(c_\mu)
                                       = q_v^{\langle-\rho,\lambda\rangle}\sum_{\mu\leq\lambda}d_\lambda(\mu)\eta(\mu(\pi_v))\Psi_{1,v}(c_\mu)\\
                                       &= \eta(\lambda(\pi_v))q_v^{\langle-\rho,\lambda\rangle}\sum_{\mu\leq\lambda}d_\lambda(\mu)\Psi_{1,v}(c_\mu)
                                       = \eta(\lambda(\pi_v))\chi_\lambda(s_{1,v}),
  \end{align*}
  after another appeal to Lemma~\ref{lem:etahat}.
Similarly, Equations~\eqref{eq:satake} can be used  to show that Equation~\eqref{eq:chi} implies Equation~\eqref{eq:psi}.
\end{proof}

\subsection{Effects of differential operators on Galois representations}\label{final-effects-section}
We assume the group $G$ is symplectic or unitary.

\begin{proof}[Proof of Theorem~\ref{charpcor}]
(For consistency with the notation in Section~\ref{AT-sec}, here $\lambda$ denotes a dominant weight in $ P^+$, and the admissible weight in the statement of Theorem~\ref{charpcor} --previously denoted by $\lambda$-- is  $\kappa_0$.)
Theorem~\ref{TandTheta} implies that the assumptions in Theorem~\ref{axiomaticthm} are satisfied by the systems of Hecke eigenvalues associated with a mod $ p$ automorphic form $f$, of weight $\kappa$,  and its image under $\Theta^{\kappa_0}=\Theta^{\kappa_0}_{\kappa}$. 
More precisely, Theorem~\ref{TandTheta} implies
\begin{align*}
\Psi_{\Theta^{\kappa_0}(f),v}(c_\lambda)&={\nu(\lambda(\pi_v))}^{|\kappa_0|/2} \Psi_{f,v} (c_\lambda),
\end{align*}
for all $\lambda\in P^+$ and  all but finitely many places $v$ (e.g., all places $v$ not above $p$ and of good reduction for the Shimura variety $\Sh$). Hence, the theorem is an immediate consequence of Theorem~\ref{axiomaticthm}. 
\end{proof}

Theorem~\ref{charpcor} could play a role in the study of Galois representations in positive characteristic, and  also in generalizations of Serre's weight conjecture.
For modular forms on $\GL_2$ over $\QQ$, the theory of $\Theta$-operators, including an analogue of Theorem~\ref{charpcor}, was used in the study of the modularity of $2$-dimensional Galois representations in positive characteristic.
Serre's weight conjecture predicted that every odd irreducible continuous $2$-dimensional mod $p$ Galois representation $\rho$ that is unramified outside $p$ has a twist by a power of the cyclotomic character that is modular of weight at most $p+1$.
Using $\Theta$-operators, it was shown that the statement also implies that $\rho$ itself is modular, of weight at most $p^2-1$. 

  On the other hand,  Theorem~\ref{charpcor} is only a first step in this direction, as more can be inferred from the existence of mod $p$ weight-raising differential operators. More precisely, the study of the image (and kernel) of the theta operators would yield a range for weights of modularity
for mod $p$ Galois representations, while  the study of theta cycles describes all weights occurring  by twisting the Galois representation by powers of the cyclotomic characters.  For modular forms, this is carried out in~\cite{Katz-theta},~\cite{jochnowitz} and~\cite{Edixhoven}; for Siegel modular forms of degree 2, we refer to~\cite{Yama}; for Hilbert modular forms, we refer to~\cite[Section 16]{AndreattaGoren}.

To exemplify potential future applications,  we consider the case of  Siegel and Hermitian modular forms of scalar, parallel weight, and of $\Theta$-operators that raise the weight by a parallel
weight. That is, we consider iterations of the operator $\Theta $, where $\Theta=\Theta^{\underline 2}$ in the Siegel case and $\Theta=\Theta^{\underline{1}}$ in the Hermitian case. For $d\geq 1$, we write $\Theta^d=\Theta\circ\cdots\circ \Theta$, $d$-times. We assume $p>n$.

\begin{remark}
By Theorem~\ref{charpcor}, $\Theta$ acts on mod $p$ Galois representations by twisting by the $n$-th power of the cyclotomic character, that is
\begin{align*}\label{EQCYCLC}
\rho_{\Theta(f)}={(\hat{\nu}\circ\chi)}^n\otimes \rho_f.
\end{align*}
In particular, when $(n,p-1)=1$,  if the twist of $\rho$ by some power of the cyclotomic character is modular, then $\rho$ itself is modular.    Also, if $\rho$ is a Galois representation associated with a mod $p$ automorphic form $f$, then
any twist of $\rho$ by powers of the cyclotomic character is  associated with the image of $f$ under some iteration of $\Theta$.  That is, if ${\left(\hat{\nu}\circ\chi\right)}^m\otimes\rho =  \rho_f$ for some modular form $f$ and integer $m\geq 1$, then $\rho = \rho_{\Theta^d(f)}$ for $d$ such that $dn \equiv m \mod p-1$.
\end{remark}

Let $f\neq 0$ be a mod $p$ Siegel or Hermitian automorphic form, of parallel weight. We write  $w(f)$ for the \emph{minimal weight} of $f$ (in the terminology of~\cite[Section I]{Katz-theta} recalled in Remark~\ref{katzfiltratiion-rmk}, the minimal weight is called the exact filtration).
By definition, $w(f)$ is  the smallest integer $k>0$ for which there exists a mod $p$ automorphic form $g$ of weight $\underline{k}$ whose (mod $p$) $q$-expansion agrees with that of $f$, that is $g(q)\equiv f(q)\mod p$. In particular, $w(f)=w(\hasse f)$.

  Suppose $f$ is in the image of $\Theta$.
We define  the \emph{$\Theta$-cycle} of $f$ as
  \[\left(w(f), w(\Theta(f)),\dots , w(\Theta^{p-2}(f))\right).\]
 By~\cite[Theorem 5.2.6]{EFMV}, for any mod $p$ automorphic form $g$, we have $\Theta(g)(q)\equiv \Theta^p(g)(q)\bmod p$. We deduce that if $f$ is a mod $p$ automorphic form in the image of the operator $\Theta$, then $\Theta(f)\neq 0$ and $w(f)=w(\Theta^{p-1}f)$.
In particular, the minimal weights
of $f$ and of its images under iterations of the operator $\Theta$ all occur in the $\Theta$-cycle of $f$.

\begin{remark}
By the $p$-adic and algebraic $q$-expansion principles (\cite[Section 8.4]{hida} and~\cite[Prop 7.1.2.15]{la}, also~\cite[Theorem 1]{Ichikawa} for Siegel modular forms  of level 1, and $p>n+2$),
if $f,g$ are two 
 mod $p$ automorphic forms, of scalar parallel weight  $k,h$ respectively,  such that $f(q)\equiv g(q)\mod p$ and $h\leq k$, then $h\equiv k\mod (p-1)$ and $f=\hasse^n g$, for $n=(k-h)/(p-1)$.
We deduce $w(f)=k-b(p-1)$, for  $b\geq 0$ the exponent of the power of $\hasse$ that exactly divides $f$.
In particular, if $f$ is a modular form of weight ${k}$,  $1\leq k\leq p-2$, then $w(f)=k$. Also, if $f$ is a cuspidal modular form of weight $k=p-1$, then $w(f)=p-1$.

Assume  $\Theta(f)\neq 0$.
By Theorem~\ref{ana_thm},  $\Theta(f)$ is an automorphic form of weight $k+2+n(p-1)$; hence, $w(\Theta(f)) = w(f)+2+(n-a)(p-1)$,
where $a\geq 0$ is the exponent of the power of $\hasse$ that exactly divides $\Theta(f)$.  By Remark~\ref{power-remark}, we expect $a\geq n-1$; by Section~\ref{BNRC-section}, the statement is true if $f$ is the mod $p$ reduction of a scalar-weight Siegel modular form of level $1$, and $p> n+2$.
Indeed, by Proposition~\ref{RC_prop},  $\Theta(f)$ is divisible by $\hasse^{n-1}$; hence, $a\geq n-1$ and
\[
  w(\Theta(f))=w(f)+p+1-b(p-1),
\]
for some integer $b\geq 0$.
(Note that if $f$ is the mod $p$ reduction of a scalar-weight Siegel modular form of level $1$,  
 the assumption $\Theta(f)\neq 0$ also implies $w(f)\geq n/2$.)

For modular forms, the integer $b\geq 0 $   is explicitly computed by Jochnowitz in~\cite{jochnowitz} (level $1$) and by Edixhoven in~\cite[\S{}3.2]{Edixhoven} (general level). In particular, they prove that $b=0$ unless $w(f)\equiv 0\mod p$.

\end{remark}
\bibliography{theta}

\end{document}